\documentclass[12pt]{article}
\usepackage{adjustbox}
\usepackage{amsmath} 
\usepackage{amsthm} 
\usepackage{amssymb} 
\usepackage[cal=boondoxo,scr=euler]{mathalfa} 
\usepackage[dvipsnames]{xcolor} 
\usepackage{mathtools} 
\usepackage{mma}
\usepackage{tikz,tikz-cd}
\usepackage[backref=page,linktocpage]{hyperref} 
\usepackage{graphicx} 
\usepackage{cleveref}
\usepackage{enumerate}
\usepackage{float}
\usepackage{setspace}
\usepackage[utf8]{inputenc}
\usepackage[T1]{fontenc}
\usepackage{caption}
\usetikzlibrary{arrows}
\usetikzlibrary{calc,positioning,arrows,decorations.pathreplacing}

\usepackage[a4paper,top=3.0cm,bottom=2.54cm,left=2.54cm,right=2.54cm]{geometry}
\usepackage{ytableau}

\DeclareMathOperator{\rk}{rk}
\newcommand{\Ind}{\mathrm{Ind}}
\newcommand{\Res}{\mathrm{Res}}
\newcommand{\ch}{\mathrm{ch}}
\newcommand{\s}{\mathfrak{S}}
\newcommand{\U}{\mathrm{U}}
\newcommand{\M}{\mathrm{M}}
\newcommand{\x}{\mathbf{x}}

\hypersetup{ 
	colorlinks = true,
	linkbordercolor = {white},
	linkcolor = {BrickRed},
	anchorcolor = {black},
	citecolor = {BrickRed},
	filecolor = {cyan},
	menucolor = {BrickRed},
	runcolor = {cyan},
	urlcolor = {black}
}

\newtheoremstyle{teoremas} %
{11pt}%
{11pt}%
{\itshape}%
{}%
{\bfseries}%
{}%
{.5em}%
{}%

\theoremstyle{teoremas} %
\newtheorem{theorem}{Theorem}[section] %
\newtheorem{corollary}[theorem]{Corollary} %
\newtheorem{lemma}[theorem]{Lemma} %
\newtheorem{proposition}[theorem]{Proposition} %
\newtheorem{conjecture}[theorem]{Conjecture} %

\newtheoremstyle{definition} %
{11pt}%
{11pt}%
{}%
{}%
{\bfseries}%
{}%
{.5em}%
{}%

\theoremstyle{definition}

\newtheorem{remark}[theorem]{Remark} %

\crefname{theorem}{theorem}{theorems} %
\Crefname{theorem}{Theorem}{Theorems} %
\crefname{lemma}{lemma}{lemmas} %
\Crefname{lemma}{Lemma}{Lemmas} %
\crefname{proposition}{proposition}{propositions} 
\Crefname{proposition}{Proposition}{Propositions} %

\makeatletter
\def\dual#1{\expandafter\dual@aux#1\@nil}
\def\dual@aux#1/#2\@nil{\begin{tabular}{@{}c@{}}#1\\#2\end{tabular}}
\makeatother

\begin{document}
\begin{center}
{\large \bf Equivariant inverse $Z$-polynomials of matroids}
\end{center}

\begin{center}
Alice L.L. Gao$^{1}$, Yun Li$^{2}$, and Matthew H.Y. Xie$^{3}$\\[6pt]

 $^{1,2}$School of Mathematics and Statistics,\\
 Northwestern Polytechnical University, Xi'an, Shaanxi 710072, P.R. China

$^{3}$College of Science, \\
 Tianjin University of Technology, Tianjin 300384, P.R. China\\[6pt]

 Email: $^{1}${\tt llgao@nwpu.edu.cn},
 $^{2}${\tt liyun091402@163.com},
 $^{3}${\tt xie@email.tjut.edu.cn}
\end{center}

\noindent\textbf{Abstract.}
Motivated by the notion of the inverse $Z$-polynomial
introduced by Ferroni, Matherne, Stevens, and Vecchi,
we study the equivariant inverse $Z$-polynomial of a matroid
equipped with a finite group.
We prove that the coefficients of the equivariant inverse $Z$-polynomials are honest representations
and that these polynomials are palindromic.
Explicit formulas are obtained for uniform matroids equipped with
the symmetric group.
The corresponding formulas for $q$-niform matroids are derived
using the Comparison Theorem for unipotent representations. For arbitrary equivariant paving matroids, explicit expressions
are obtained by relating the polynomials of a matroid
to those of its relaxation.
We show that these polynomials are equivariantly unimodal and
strongly inductively log-concave for both uniform and $q$-niform matroids.
Motivated by the properties of equivariant $Z$-polynomials,
we conjecture that the coefficients of the equivariant inverse $Z$-polynomials are equivariantly unimodal
and strongly equivariantly log-concave.

\noindent \emph{AMS Classification 2020:}
05B35, 05E05, 20C30


\noindent \emph{Corresponding Author: Matthew H.Y. Xie}

\section{Introduction}

The objective of this paper is to study the equivariant 
inverse $Z$-polynomials of matroids.
The Kazhdan--Lusztig polynomials of matroids, first introduced by 
Elias, Proudfoot, and Wakefield~\cite{elias2016kazhdan}, 
have attracted considerable attention in recent years. Proudfoot, Xu, and Young~\cite{proudfoot2018z} introduced 
the $Z$-polynomial $Z_\M(t)$ of a matroid $\M$, defined as a 
weighted sum of the Kazhdan--Lusztig polynomials of all 
contractions of $\M$. 
Braden, Huh, Matherne, Proudfoot, and Wang~\cite{braden2020singular} 
provided a cohomological interpretation of $Z_\M(t)$,
where $Z_\M(t)$ is realized as the Hilbert series of the
intersection cohomology module associated with $\M$.
Within the
Kazhdan--Lusztig--Stanley framework,
Ferroni, Matherne, Stevens, and Vecchi~\cite{ferroni2024hilbert}
introduced
the inverse $Z$-polynomial $Y_\M(t)$ of a matroid $\M$.
Gao, Ruan, and Xie \cite{gao2025inverse} further studied the fundamental properties of $Y_\M(t)$, and 
Braden, Ferroni, Matherne, and Nepal~\cite{braden2025invkldeletion} derived a deletion formula for this polynomial.

When matroid $\M$ is equipped with a finite group $W$,
one can define the equivariant Kazhdan--Lusztig polynomial~\cite{gedeon2017equivariant} and the equivariant $Z$-polynomial~\cite{proudfoot2018z}, whose coefficients are honest (rather than virtual) representations of $W$, and taking the dimensions of these representations recovers the corresponding non-equivariant polynomials. 
Within the framework of equivariant Kazhdan--Lusztig--Stanley 
theory, one can similarly define the equivariant inverse 
$Z$-polynomial. 
In this paper, we compute the equivariant inverse $Z$-polynomials
of uniform matroids equipped with the symmetric group,
and of $q$-niform matroids equipped with the general
linear group.
We further obtain explicit formulas for the equivariant inverse
$Z$-polynomials of arbitrary equivariant paving
matroids.
As applications, we establish the equivariant unimodality and strongly
induced log-concavity of these polynomials for uniform and
$q$-niform matroids.

We begin by recalling several fundamental notions concerning matroids, following the notation of Ferroni, Matherne, Stevens, and Vecchi~\cite{ferroni2024hilbert}.
A matroid $\M = (E, \mathscr{B})$ consists of a finite ground set~$E$
and a collection~$\mathscr{B}$ of subsets of~$E$, called its bases.
For any subset $A \subseteq E$, the rank function of~$\M$ is defined by
$\rk_{\M}(A) = \max_{B \in \mathscr{B}} \lvert A \cap B \rvert$,
and the rank of $\M$ is $\rk(\M) = \rk_{\M}(E)$.
A subset $F \subseteq E$ is called a flat if no element of
$E \setminus F$ can be added to $F$ without increasing its rank.
The collection of all flats of $\M$, ordered by inclusion,
forms a lattice. We denote it by $\mathcal{L}(\M)$.
For any subset $A \subseteq E$, we write $\M|_{A}$ and $\M/A$ for the restriction and contraction of $\M$ at $A$, respectively.

Based on these basic notions, we now recall the definition of
equivariant matroids and associated polynomial invariants.
Let $W$ be a finite group 
and let $\mathrm{VRep}(W)$ denote the Grothendieck ring of virtual finite dimensional complex representations of $W$. Let $\mathrm{Rep}(W)\subseteq \mathrm{VRep}(W)$ denote its subsemiring of honest representations.
We define the graded $\mathbb{Z}[t]$-module
$\mathrm{grVRep}(W)
 := \mathrm{VRep}(W) \otimes_{\mathbb{Z}} \mathbb{Z}[t].$
Given an equivariant matroid $W \curvearrowright \M$, let $Q^{W}_{\M}(t)$ and $\mu^{W}_{\M}$ denote the equivariant inverse Kazhdan--Lusztig polynomial and the equivariant M\"{o}bius invariant, respectively. We define the equivariant inverse 
$Z$-polynomial of $W \curvearrowright \M$ as an element of 
$\mathrm{grVRep}(W)$ by
\begin{align}\label{YM}
 Y^{W}_{\M}(t)
 := (-1)^{\rk(\M)} \sum_{[F] \in \mathcal{L}(\M)/W}
 (-1)^{\rk(F)} t^{\rk(\M/F)}
 \,\mathrm{Ind}^{W}_{W_F}\!
 \big(Q^{W_F}_{\M|_{F}}(t) \otimes \mu^{W_F}_{\M/F}\big),
\end{align}
where $W_F \subseteq W$ denotes the stabilizer of $F$, and the sum 
runs over the $W$-orbits of flats of $\M$.
Note that taking dimensions of the representations in \eqref{YM} recovers the ordinary inverse $Z$-polynomial $Y_{\M}(t)$.

Since the equivariant $Z$-polynomials and equivariant inverse $Z$-polynomials 
of matroids are defined in an analogous manner, it is natural to ask 
whether they share similar properties. 
We show that the coefficients of the equivariant inverse
$Z$-polynomials are honest representations, as established in
Proposition~\ref{equ-inv-z-hone}, and are palindromic, as proved in
Proposition~\ref{equ-inv-z-pro}. These properties are analogous to those
of the equivariant $Z$-polynomials.

 Let $\U_{k,n}$ denote the uniform matroid of rank $k$ on $n$ elements. We compute the equivariant inverse 
$Z$-polynomial of the uniform matroid $\U_{k,n}$ equipped with the symmetric group $\s_n $, which is stated as follows.

\begin{theorem}\label{thm-equi-uni}
For the equivariant uniform matroid $\s_n \curvearrowright \U_{k,n}$ with $n\geq k \geq 1$,
\begin{align}\label{eq-ekl-uniform}
Y_{\U_{k,n}}^{\s_n}(t)
&= \sum_{i=0}^{\lfloor k/2 \rfloor}
\Ind_{\s_i \times \s_{n-i}}^{\s_n}\!\bigl(
V_{(1^i)} \otimes V_{(n-k+1,1^{k-i-1})}
\bigr)\, t^{i}\nonumber\\
&\qquad + \sum_{i=0}^{\lfloor (k-1)/2\rfloor}
\Ind_{\s_i \times \s_{n-i}}^{\s_n}\!\bigl(
V_{(1^i)} \otimes V_{(n-k+1,1^{k-i-1})}
\bigr)\, t^{k-i},
\end{align}
where $V_{\lambda}$ denotes the irreducible representation of $\mathfrak{S}_{|\lambda|}$ indexed by $\lambda$, and we set $V_{\lambda} = 0$ if $\lambda$ is not a valid partition.
\end{theorem}

With these explicit formulas established, we study the unimodality and log-concavity for the coefficients of the equivariant inverse 
$Z$-polynomials.
Let $\{C_i\}_{i=0}^n$ be a sequence of virtual representations of a finite group $W$.
We say that this sequence is equivariantly unimodal if there exists an index $i$ such that
\begin{align*}
 C_0 \le_W C_1 \le_W \cdots \le_W C_i \ge_W \cdots \ge_W C_n,
\end{align*}
where $A \le_W B$ means that $B-A \in \mathrm{Rep}(W)$.
The sequence is said to be equivariantly log-concave if
\begin{align*}
 C_i \otimes C_i - C_{i-1} \otimes C_{i+1} \in \mathrm{Rep}(W) \qquad \text{for~ all~} 1 \le i \le n-1.
\end{align*}
Applying the dimension map to these inequalities yields the unimodality and log-concavity of the sequence $\{\dim C_i\}_{i=0}^n$.
In the non-equivariant setting, log-concavity implies unimodality, whereas this implication fails in the equivariant setting.
For example, the equivariant Kazhdan--Lusztig polynomial of the uniform matroid $\U_{4,5}$ is given by $ P^{\s_{5}}_{\U_{4,5}}(t) = V_{(5)} + V_{(3,2)} t $
, see \cite[Theorem 3.1]{gedeon2017equivariant}.
This polynomial is not equivariantly unimodal, although it is shown in
\cite[Remark~5.5]{gedeon2017equivariant} to be equivariantly log-concave.
Hence, equivariant unimodality is a property of independent interest.

Equivariant unimodality has attracted considerable attention.
Angarone, Nathanson, and Reiner~\cite{angarone2025chow} showed that the graded pieces of the Chow ring of a matroid form an equivariantly unimodal sequence.
Gui and Xiong~\cite{gui2024equivariant} highlighted the role of
equivariant unimodality in the study of equivariant log-concavity
conjectures.
In equivariant Ehrhart theory, Elia, Kim, and Supina~\cite[Question~4.10]{elia2024techniques} asked which polytopes exhibit equivariant unimodality for their $h^\ast$-vectors. Braden, Huh, Matherne, Proudfoot, and Wang~\cite{braden2020singular} proved that the coefficients of the equivariant $Z$-polynomial are equivariantly unimodal.
This result motivates the following conjecture for the equivariant
inverse $Z$-polynomial.

\begin{conjecture}\label{conj-unimodal}
For any matroid $\M$ equipped with equipped with an action of a finite group $W$, the coefficients of its equivariant inverse $Z$–polynomial are equivariantly unimodal.
\end{conjecture}

We confirm this conjecture for uniform matroids.

\begin{theorem}\label{thm-invz-uni-unimodal}
Let $n\ge k\ge 1$ be integers. The coefficients of the equivariant inverse $Z$-polynomial of $\s_n \curvearrowright \U_{k,n}$ are equivariantly unimodal.
\end{theorem}

A stronger property, called strong equivariant log-concavity, was introduced in \cite{gedeon2017equivariant}.
A sequence $\{C_i\}_{i=0}^n$ of virtual representations of $W$
is said to be strongly equivariantly log-concave if
\begin{align*}
 C_i \otimes C_j - C_{i-k} \otimes C_{j+k} \in \mathrm{Rep}(W), \qquad \text{for~ all~}1\le k\le i\le j \le n-1.
\end{align*}
Proudfoot, Xu, and Young~\cite{proudfoot2018z} conjectured that the coefficients of the equivariant $Z$-polynomials are strongly equivariantly log-concave.
Motivated by this conjecture, we propose the following analogue for
equivariant inverse $Z$-polynomials.

\begin{conjecture}\label{conj-log}
For any matroid $\M$ equipped with an action of a finite group $W$, the coefficients of the equivariant inverse $Z$–polynomial are strongly equivariantly log-concave.
\end{conjecture}

When $W=\mathfrak{S}_n$, establishing strong equivariant log-concavity
requires decomposing the Kronecker products of Schur functions.
Such decompositions are difficult and a general
combinatorial interpretation of the Kronecker coefficients remains a central
open problem in algebraic combinatorics.
For this reason, we instead work with 
strongly induced log-concavity introduced by Gao
et al.~\cite{gao2023induced}.
If there exist a finite group $G$, a subgroup $H\le G$, and a
homomorphism $\phi\colon H\to W\times W$ such that
$$
\Ind_H^G\bigl(C_i\boxtimes C_j\bigr)
-
\Ind_H^G\bigl(C_{i-1}\boxtimes C_{j+1}\bigr)
\in \mathrm{Rep}(G),
\quad\text{for any } 1\le i\le j\leq n-1,
$$
where the external tensor product $C_i\boxtimes C_j$ is regarded as a
representation of $H$ via pullback along $\phi$, then the sequence
$\{C_i\}_{i= 0}^n$ is said to be strongly inductively log-concave with respect to
$G$, $H$, and $\phi$.
Applying the dimension map to these inequalities also yields the strong log-concavity of the sequence $\{\dim C_i\}_{i=0}^n$.

Based on Theorem \ref{thm-equi-uni}, we obtain the following result.

\begin{theorem}\label{thm-invz-uni-unimodal-log}
Let $n\ge k\ge 1$ be integers. The coefficients of the equivariant inverse $Z$-polynomial $Y_{\U_{k,n}}^{\s_n}(t)$ are strongly inductively log-concave.
\end{theorem}

We also consider the $q$-niform matroid $\U_{k,n}(q)$ equipped with the natural action of the general linear group $\mathrm{GL}_n(\mathbb{F}_q)$. Using the Comparison Theorem for unipotent representations, we prove the corresponding results for $q$-niform matroid in Subsection \ref{Q-niform}, which are in parallel with those for uniform matroid.

To compute the equivariant inverse $Z$-polynomials of arbitrary equivariant paving
matroids, we use the operation of relaxation.
For a matroid $\M=(E,\mathscr{B})$ of size $n$ and rank $k$,
		a subset $I \subseteq E$ is independent if it is contained
		in some basis $B \in \mathscr{B}$, and dependent otherwise.
		A circuit is a minimal dependent set.
		Following Ferroni, Nasr, and Vecchi~\cite{ferroni2023stressed}, a stressed hyperplane $H$ of $\M$ is a flat of rank $k-1$ in which every $k$-subset is a circuit.
		The relaxation of $\M$ along $H$ is obtained by adjoining
		to the set of bases all $k$-subsets of $H$.
	Ferroni, Nasr, and Vecchi~\cite{ferroni2023stressed} proved that a matroid is paving if and only if relaxing all its stressed
	hyperplanes of size at least $k$ yields the uniform matroid $\U_{k,n}$.
	
 In the equivariant setting, Karn, Nasr, Proudfoot, and Vecchi~\cite{karn2023equivariant} relaxed entire $W$-orbits
	of stressed hyperplanes simultaneously. Building
	on the stressed hyperplane analysis of Ferroni, Nasr, and Vecchi~\cite{ferroni2023stressed}, they derived explicit relaxation formulas for the equivariant Kazhdan--Lusztig polynomial
	and the equivariant inverse Kazhdan--Lusztig polynomial.
	They also proved the corresponding statement for the equivariant $Z$-polynomial, but did not provide an explicit formula since the resulting expression is less elegant.
	This makes it possible to transport formulas for equivariant uniform matroids to
	explicit expressions for equivariant paving matroids.

We prove the
analogous relaxation formula for equivariant inverse $Z$-polynomials, which is stated in the following
theorem.

\begin{theorem}\label{thm-paving}
Let $\M$ be a matroid of rank $k$ with an action
of a finite group $W$, and let $H$ be any stressed
hyperplane of size $h$ with stabilizer $W_H \subseteq W$.
If $\widetilde{\M}$ is obtained by simultaneously
relaxing the entire $W$–orbit of $H$, then
\begin{align} \label{eq-paving-thm}
Y^{W}_{\widetilde{\M}}(t) - Y^{W}_{\M}(t)
= \mathrm{Ind}^{W}_{W_H} \, \mathrm{Res}^{\s_h}_{W_H} \,
\big(Y^{\s_h}_{\U_{k,h}}(t)
- \frac{1+(-1)^k}{2} V_{(h-k+2,2^{\frac{k}{2} -1})} t^{\frac{k}{2}} \big),
\end{align}
where $\mathrm{Res}^{\s_h}_{W_H}$ denotes the functor pulling back a representation along the natural homomorphism $W_H \to \s_h$.
\end{theorem}

Taking the dimensions of the representations in
\eqref{eq-paving-thm} yields the following expression
for the ordinary inverse $Z$-polynomial $Y_{\M}(t)$ of paving matroids.
This formula
differs in form from the one previously obtained in \cite[Proposition~5.3]{gao2025inverse}.

\begin{theorem}\label{cor-paving-or-z}
Let $\M$ be a paving matroid of rank $k$ and cardinality $n$. Then
\begin{align*}
 Y_{\M}(t)=Y_{\U_{k,n}}(t)-\sum_{h \ge k}\lambda_h \Big(Y_{\U_{k,h}}(t)- \frac{2h\big(1+(-1)^k\big)}{(2h-k)(2h-k+2)} \dbinom{h-1}{\frac{k}{2},\frac{k}{2}-1,h-k}t^{\frac{k}{2}} \Big),
\end{align*}
where $\lambda_h$ denotes the number of stressed hyperplanes of size $h$ in $\M$.
\end{theorem}

An unpublished conjecture of Gedeon asserts that the coefficients of the
equivariant Kazhdan--Lusztig polynomial of a matroid $\M$
are bounded above by those of the equivariant
Kazhdan--Lusztig polynomial of the uniform matroid
of the same rank on the same ground set.
A precise formulation of this conjecture is given in
\cite[Conjecture~1.6]{karn2023equivariant},
where it is shown to hold when $\M$ is paving.
The same work also establishes an analogous statement
for equivariant inverse Kazhdan--Lusztig polynomials.
Here we show that this property extends to
equivariant inverse $Z$-polynomials.

\begin{theorem}\label{thm-uni-bigest}
Let $\M$ be a paving matroid of rank $k$ on the ground set $E$, and let $W$ be a finite group
acting on $E$ and preserving $\M$.
Let $\U_{k,E}$ denote the uniform matroid of rank $k$ on $E$.
Then $Y^{W}_{\M}(t)$ is coefficientwise bounded above by $Y^{W}_{\U_{k,E}}(t)$. That is, the coefficients of 
$Y^{W}_{\U_{k,E}}(t)-Y^{W}_{\M}(t)$ are honest representations of $W$.

\end{theorem}

This paper is organized as follows.
In Section~\ref{section-equi-invz}, we establish two fundamental
properties of equivariant inverse $Z$-polynomials.
We prove that their coefficients are honest representations and that the
polynomials are palindromic.
In Section~\ref{Uniform matroids}, we derive explicit formulas for
the equivariant inverse $Z$-polynomials of uniform matroids equipped with the natural action of the symmetric group using
symmetric function theory. We then extend these results to
$q$-niform matroids equipped with the general linear
group.
As applications, we prove that these polynomials are equivariantly
unimodal and strongly inductively log-concave.
In Section~\ref{subse-paving}, we relate the equivariant inverse
$Z$-polynomials of a matroid $\M$ to those of its relaxation
$\widetilde{\M}$. This relation allows us to compute the equivariant
inverse $Z$-polynomials of any equivariant paving matroid.


\section{Equivariant inverse \texorpdfstring{$Z$-polynomials}{Z-polynomials}}\label{section-equi-invz}

In this section, we establish two fundamental properties of the equivariant
inverse $Z$-polynomials of matroids.

Braden, Huh, Matherne, Proudfoot, and Wang \cite{braden2020singular}
showed that if a matroid is equipped with an action of a finite group $W$, then the coefficients of its equivariant $Z$-polynomial are honest representations of $W$.
We prove that an analogous statement holds for the equivariant inverse $Z$-polynomials.

\begin{proposition}\label{equ-inv-z-hone}
For any equivariant matroid $W \curvearrowright \M$, the coefficients of $Y^{W}_{\M}(t)$ are honest representations of $W$.
\end{proposition}

\begin{proof}
By \cite{gedeon2017equivariant}, the equivariant characteristic polynomial of $W \curvearrowright \M$ is defined by
\begin{align}\label{char-uni-2}
 H_\M^W(t):
 = \sum_{i=0}^{\rk(\M)} (-1)^i\, t^{\rk(\M)-i}\, OS^W_{\M,i},
\end{align}
where $OS^W_{\M,i}$ is the degree $i$ part of the
Orlik--Solomon algebra, regarded as an honest representation of $W$
arising from the natural action of $W$ on $\M$.
On the other hand, \cite[Proposition~4.1]{proudfoot2021equivariant}
gives the alternative expression
\begin{equation}\label{char-uni-1}
 H^{W}_{\M}(t)
 = \sum_{[F]\in \mathcal{L}(\M)/W}
 \Ind_{W_F}^{W}\!\left(
 \mu_{\M|_F}^{W_F} \otimes
 t^{\rk(\M)-\rk(F)}\,\zeta_{\M/F}^{W_F}
 \right),
\end{equation}
where $\zeta_{\M}^{W}$ denotes the trivial
representation of $W$ in degree~$0$.
Comparing the constant terms in \eqref{char-uni-2} and \eqref{char-uni-1} yields
\begin{equation}\label{char-uni-3}
 \mu_{\M}^{W} = (-1)^{\rk(\M)}\, OS^{W}_{\M, \, \rk(\M)}.
\end{equation}
Substituting \eqref{char-uni-3} into the defining expression
\eqref{YM} of $Y^W_{\M}(t)$, we see that the signs
$(-1)^{\rk(\M)}$, $(-1)^{\rk(F)}$, and
$(-1)^{\rk(\M/F)}$ cancel.
It follows that each summand of $Y^W_{\M}(t)$ can be written as
$$
t^{\rk(\M/F)}
\Ind^W_{W_F}\!\bigl(
Q^{W_F}_{\M|F}(t)
\otimes
OS^{W_F}_{\M/F, \rk(\M/F)}
\bigr).
$$
Since both $Q^{W_F}_{\M|F}(t)$ and
$OS^{W_F}_{\M/F,\rk(\M/F)}$ have honest coefficients,
and this property is preserved under tensor product and induction,
all coefficients of $Y^W_{\M}(t)$ are honest
representations of $W$, as required.
\end{proof}

We now consider the palindromicity of equivariant inverse
$Z$-polynomials.
It is well known that the equivariant $Z$-polynomials are palindromic,
as shown in \cite[Section~6]{proudfoot2018z} and
\cite[Corollary~4.1]{proudfoot2021equivariant}.
The same property holds for equivariant inverse
$Z$-polynomials.
This follows directly from the equivariant
Kazhdan--Lusztig--Stanley theory developed in
\cite{proudfoot2021equivariant}, and we omit the details.

\begin{proposition}\label{equ-inv-z-pro}
Let $W \curvearrowright \M$ be an equivariant matroid. 
Then $Y^{W}_{\M}(t)$ is palindromic of degree 
$\rk(\M)$, that is,
\begin{align*}
 t^{\rk(\M)}Y^{W}_{\M}(t^{-1})=Y^{W}_{\M}(t).
\end{align*}
\end{proposition}

\section{Uniform and \texorpdfstring{$q$}{q}-niform matroids}\label{Uniform matroids}

In this section, we prove 
Theorems~\ref{thm-equi-uni}, 
\ref{thm-invz-uni-unimodal-log}, 
and \ref{thm-invz-uni-unimodal}.
In Subsection~\ref{expression-inverse-Z}, we derive 
an explicit formula for the equivariant inverse 
$Z$-polynomials of uniform matroids equipped with 
the action of the symmetric group. 
In Subsection~\ref{sec-unimodal-log}, we show that the 
coefficients of these polynomials are equivariantly unimodal and strongly inductively log-concave. 
In Subsection~\ref{Q-niform}, we obtain explicit formulas 
for the equivariant inverse $Z$-polynomials of 
$q$-niform matroids using the Comparison Theorem for 
unipotent representations, 
and we then establish that these polynomials are also
equivariantly unimodal and strongly inductively 
log-concave.

\subsection{Equivariant inverse \texorpdfstring{$Z$}{Z}-polynomials of uniform matroids}\label{expression-inverse-Z}

We begin by describing the equivariant M\"obius invariants of
uniform matroids, which will be used in the proof of
Theorem~\ref{thm-equi-uni}.

\begin{lemma}\label{lem-mobius-equi-uni}
Let $\U_{k,n}$ be a uniform matroid with 
$n \ge k \ge 1$. Then its equivariant M\"{o}bius invariant is
\begin{align}\label{equi-uni-mobius}
\mu_{\U_{k,n}}^{\s_n} = (-1)^k V_{(n-k+1,1^{k-1})} .
\end{align}
\end{lemma}

\begin{proof}
Applying \eqref{char-uni-3} to the equivariant matroid $\s_n \curvearrowright \U_{k,n}$
and invoking
\cite[Proposition~3.9]{gedeon2017equivariant}, which asserts that
\begin{align*} OS^{\mathfrak{S}_n}_{\U_{k,n}, k}
 = V_{(n-k+1,1^{ k-1})},
 \qquad \qquad \text{for~} n\ge k\ge 1,
\end{align*}
gives the claimed identity.
\end{proof}

We are now in the position to prove Theorem~\ref{thm-equi-uni}.

\begin{proof}[Proof of Theorem~\ref{thm-equi-uni}]
Applying \eqref{YM} to $\s_n\curvearrowright \U_{k,n}$ gives
\begin{align}\label{eq-ui-qe-ninv-z}
Y_{\U_{k,n}}^{\s_n}(t)
=(-1)^{k}\sum_{[F]\in \mathcal{L}(\U_{k,n})/\s_n} (-1)^{\rk (F)} t^{\rk(\U_{k,n}/F)} \mathrm{Ind}_{(\s_n)_F}^{\s_n}
\left(Q_{\U_{k,n}|_F}^{(\s_n)_F}(t) \otimes \mu_{\U_{k,n}/F}^{(\s_n)_F} \right).
\end{align}
If $F$ has rank $i<k$, then
\begin{align*}
\U_{k,n}|_F \cong B_i, \qquad \U_{k,n}/F \cong \U_{k-i,n-i}, \qquad (\s_n)_F \cong \s_i \times \s_{n-i},
\end{align*}
where $B_i$ is the Boolean matroid of rank $i$.
The unique flat of rank $k$ is the ground set $E$, and for $F=E$ we have
\begin{align*}
\U_{k,n}|_F = \U_{k,n}, \qquad \U_{k,n}/F \cong B_{0},
\qquad (\s_n)_F=\s_n.
\end{align*}
Substituting these flats into \eqref{eq-ui-qe-ninv-z} yields
\begin{align}\label{invz-1}
 Y_{\U_{k,n}}^{\s_n}(t)
 =(-1)^{k}\sum_{i=0}^{k-1}(-1)^{i}\mathrm{Ind}_{\s_i \times \s_{n-i}}^{\s_n} \Big( Q_{B_i}^{\s_i }(t) \otimes \mu_{\U_{k-i,n-i}}^{ \s_{n-i}}\Big)t^{k-i} + Q_{\U_{k,n}}^{\s_{n}}(t).
\end{align}
By \cite[Theorem 3.2]{gao2022equivariant},
\begin{align} \label{boo-uni-invkl2}
 Q_{\U_{k,n}}^{\s_n}(t)
 = \sum_{i=0}^{\lfloor (k-1)/2 \rfloor}
 V_{(n-k+1,2^i,1^{k-2i-1})}t^{i}, \qquad \text{~for~} n \geq k \geq 1,
\end{align}
and in particular,
\begin{align}\label{boo-uni-invkl1}
Q_{B_n}^{\s_n}(t) = V_{(1^n)}, \qquad \text{~for~ } n \geq 0.
\end{align}
Combining \eqref{invz-1} with \eqref{boo-uni-invkl2},
\eqref{boo-uni-invkl1}, and Lemma~\ref{lem-mobius-equi-uni} gives
\begin{align*}
 Y_{\U_{k,n}}^{\s_n}(t)
 &= \sum_{i=0}^{k-1}\mathrm{Ind}_{\s_i \times \s_{n-i}}^{\s_n} \left( V_{(1^i)} \otimes V_{(n-k+1,1^{k-i-1})} \right)t^{k-i}
 + \sum_{i=0}^{\lfloor (k-1)/2 \rfloor}V_{(n-k+1,2^i,1^{k-2i-1})}t^{i}.
\end{align*}
Since $Y_{\U_{k,n}}^{\s_n}(t)$ is palindromic of degree $k$, we have
$Y_{\U_{k,n}}^{\s_n}(t)
 = t^k Y_{\U_{k,n}}^{\s_n}(t^{-1})
$
 and hence
 \begin{align*}
 Y_{\U_{k,n}}^{\s_n}(t)
 = \sum_{i=0}^{k-1}\mathrm{Ind}_{\s_i \times \s_{n-i}}^{\s_n} \left( V_{(1^i)} \otimes V_{(n-k+1,1^{k-i-1})} \right)t^{i}
 + \sum_{i=0}^{\lfloor (k-1)/2 \rfloor}V_{(n-k+1,2^i,1^{k-2i-1})}t^{k-i}.
\end{align*}
	Since
	$\lfloor k/2 \rfloor <k- \lfloor (k-1)/2 \rfloor$,
	the coefficient of $t^{i}$ for $0\leq i\leq \lfloor k/2 \rfloor$ in $Y_{\U_{k,n}}^{\s_n}(t)$ is exactly
	$\mathrm{Ind}_{\s_i \times \s_{n-i}}^{\s_n} \left( V_{(1^i)} \otimes V_{(n-k+1,1^{k-i-1})} \right),$
		which completes the proof.
\end{proof}

We now describe the irreducible decomposition of the induced 
representations that appear in Theorem~\ref{thm-equi-uni}, which will 
be used in the following subsection to verify their 
equivariant unimodality.

To this end, 
we briefly recall 
some basic notions from the theory of symmetric functions; 
see \cite{stanley1999enumerative} for any undefined terminology. 
Let $\Lambda_n$ denote the $\mathbb{Z}$-module of homogeneous symmetric
functions of degree $n$ in the variables
$\mathbf{x}=(x_1,x_2,\dots)$.
The Frobenius characteristic map is an isomorphism
$$
 \ch : \mathrm{VRep}(\mathfrak{S}_n) \longrightarrow \Lambda_n,
$$
which sends the irreducible representation $V_\lambda$ to the Schur
function $s_\lambda(\x)$ for each partition $\lambda$ of $n$.
This map extends directly to the graded setting
$\mathrm{grVRep}(\mathfrak{S}_n)$.
Moreover, for graded virtual representations $V_1 \in \mathrm{grVRep}(\mathfrak{S}_{n_1})$ and $V_2 \in \mathrm{grVRep}(\mathfrak{S}_{n_2})$, we have
$$
 \ch\!\left(
\Ind_{\mathfrak{S}_{n_1}\times \mathfrak{S}_{n_2}}^{\mathfrak{S}_{n_1+n_2}}
 (V_1 \otimes V_2)
 \right)
 = \ch(V_1)\, \ch(V_2).
$$
With these preliminaries, we are ready to prove 
Proposition~\ref{prop-equi-uni}.

\begin{proposition}\label{prop-equi-uni}
For the equivariant uniform matroid $\s_n \curvearrowright \U_{k,n}$ with $n\geq k \geq 1$,
\begin{align}\label{eq-ekl-uniform-1}
 Y_{\U_{k,n}}^{\s_{n}}(t)
 = & \sum_{i=0}^{\lfloor k/2 \rfloor}
 \sum_{x=0}^{i}
 \big( V_{(n-k+1,2^{x},1^{k-2x-1})} + V_{(n-k+2,2^{x-1},1^{k-2x})}
 \big) t^{i}\nonumber\\
 &\quad +\sum_{i=0}^{\lfloor (k-1)/2 \rfloor}
 \sum_{x=0}^{i}\big( V_{(n-k+1,2^{x},1^{k-2x-1})} + V_{(n-k+2,2^{x-1},1^{k-2x})}
 \big) t^{k-i}.
\end{align}
\end{proposition}

\begin{proof}
By Theorem~\ref{thm-equi-uni} and the Frobenius
characteristic map, it suffices to show that for each
$0\le i\le \lfloor k/2\rfloor$,
\begin{align}
 s_{(1^i)}(\x) s_{(n-k+1,1^{k-i-1})}(\x)
 &=\sum_{x=0}^{i} \left(s_{(n-k+1,2^{x},1^{k-2x-1})}(\x)
 + s_{(n-k+2,2^{x-1},1^{k-2x})}(\x) \right)\nonumber\\
 &=\sum_{x=0}^{i} s_{(n-k+1,2^{x},1^{k-2x-1})}(\x)
 + \sum_{x=0}^{i-1} s_{(n-k+2,2^{x},1^{k-2x-2})}(\x).\label{uni-invz-thm-1}
\end{align}
When $n=k$, the Pieri rule \cite[Section~7]{stanley1999enumerative} implies
$$
s_{(1^i)}(\x) s_{(1^{k-i})}(\x)
= \sum_{x=0}^{i} s_{(2^x,1^{k-2x})}(\x),
$$
which agrees with the right-hand side of \eqref{uni-invz-thm-1}.
Now suppose $n>k$.
Applying the Pieri rule gives
\begin{align}
 &s_{(1^i)}(\x) s_{(n-k+1,1^{k-i-1})}(\x) \nonumber \\
 & \qquad = \sum_{x=0}^{\min\{k-i-1,i\}} s_{(n-k+1,2^{x},1^{k-2x-1})}(\x)
 + \sum_{x=0}^{\min\{k-i-1,i-1\}} s_{(n-k+2,2^{x},1^{k-2x-2})}(\x).\label{pieri-1}
 \end{align}
Since $0\le i\le\lfloor k/2\rfloor$, we have $k\ge 2i$, and therefore
$k-i-1\ge i-1$ with equality only when $i=k/2$.
Thus $\min\{k-i-1,i-1\}=i-1$,
so the upper bound of the second sum in
\eqref{pieri-1} equals $i-1$.

For the first minimum $\min\{k-i-1,i\}$ there are two cases.
If $k$ is odd, then $k-i-1\ge i$ for every $i\le\lfloor k/2\rfloor$, so
$\min\{k-i-1,i\}=i$.
If $k$ is even, the same inequality holds whenever $i<\frac{k}{2}$.
When $i=k/2$, however, one obtains $k-i-1=i-1<i$, so $\min\{k-i-1,i\}=i-1$.
In this boundary case $i=k/2$, the would-be term with $x=i$ in the first
sum of \eqref{pieri-1} corresponds to the partition
$(n-k+1,2^{i},1^{k-2i-1})$, but $k-2i-1=-1$. This partition is
not defined, and hence the corresponding Schur function vanishes by convention.
Therefore one may safely extend the first summation index
to $i$ in all cases, since the extra term is zero precisely when
$k$ is even and $i=k/2$.
Consequently, in all cases the first sum may be written with upper index
$i$, while the second sum has upper index $i-1$.
This is precisely the right-hand side of \eqref{uni-invz-thm-1}, which
completes the proof.
\end{proof}


\subsection{Equivariant unimodality and strongly induced log-concavity }\label{sec-unimodal-log}

This subsection is devoted to the proofs of 
Theorems~\ref{thm-invz-uni-unimodal} and 
\ref{thm-invz-uni-unimodal-log}. 
We begin with the equivariant unimodality of 
the equivariant inverse $Z$-polynomial of $\s_n \curvearrowright \U_{k,n}$.

\begin{proof}[Proof of Theorem~\ref{thm-invz-uni-unimodal}]
By the palindromicity of
$Y_{\U_{k,n}}^{\s_{n}}(t)$,
it suffices to show that for all integers
$n \ge k \ge 1$ and
$1 \le i \le \lfloor k/2 \rfloor$, the difference
\begin{align*}
[t^{i}]Y_{\U_{k,n}}^{\s_{n}}(t) - [t^{i-1}]Y_{\U_{k,n}}^{\s_{n}}(t) \in \mathrm{Rep}(\s_n).
\end{align*}
	Using the explicit formula of
$Y_{\U_{k,n}}^{\s_{n}}(t)$ from
	Proposition~\ref{prop-equi-uni},
	we compute for $1\leq i\leq \lfloor k/2 \rfloor$,
\begin{align*}
 [t^{i}]Y_{\U_{k,n}}^{\s_{n}}(t) - [t^{i-1}]Y_{\U_{k,n}}^{\s_{n}}(t)
 = V_{(n-k+1,2^{i},1^{k-2i-1})} + V_{(n-k+2,2^{i-1},1^{k-2i})}.
\end{align*}
Both summands are irreducible representations of $\mathfrak{S}_n$, 
so their sum belongs to $\mathrm{Rep}(\mathfrak{S}_n)$. 
This completes the proof.
\end{proof}

We now turn to the strongly induced log-concavity of the equivariant inverse $Z$-polynomial of $\s_n \curvearrowright \U_{k,n}$.
To this end, we begin with an inequality on Schur positivity.
Recall that a symmetric function $f(\x)$ is said to be Schur positive
if it can be expressed as a nonnegative linear combination of Schur functions.
We write $f \ge_s g$ if $f-g$ is Schur positive.
The following Schur positivity inequality
will be used repeatedly in what follows and is recalled here for convenience.

\begin{lemma}[{\cite[Theorem 3.1]{bergeron2006inequalities}}]\label{lem-schur-log-1}
Let $m,j,u,v$ be nonnegative integers with $1 \le m \le j$.
\begin{itemize}
\item[$(1)$] If $0 \le u < v$, then
\begin{align}\label{schur-posi-1}
 s_{(m,1^u)}(\x)\, s_{(j,1^v)}(\x)
 \le_s
 s_{(m,1^{v-1})}(\x)\, s_{(j,1^{u+1})}(\x);
 \end{align}
 \item[$(2)$] If $0 \le v \le u$, then
\begin{align}\label{schur-posi-2}
 s_{(m,1^u)}(\x)\, s_{(j,1^v)}(\x)
 \le_s
s_{(m,1^{v})}(\x)\, s_{(j,1^{u})}(\x).
 \end{align}
\end{itemize}
\end{lemma}

We provide a Schur positivity inequality, which will play a central role
in establishing the strongly induced log-concavity of
$Y_{\U_{k,n}}^{\mathfrak{S}_n}(t)$.

\begin{lemma}\label{lemma-strong-induced}
Let $n$ and $k$ be nonnegative integers with $n \ge k \ge 1$.
\begin{itemize}
 \item[$(1)$] If $1 \le i \le\lfloor k/2 \rfloor $, then
 \begin{align}\label{strong-induced-1}
 s_{(1^i)}(\x)s_{(n-k+1,1^{k-i-1})}(\x) \ge_{s} s_{(1^{i-1})}(\x)s_{(n-k+1,1^{k-i})}(\x).
 \end{align}
 \item[$(2)$] If $1 \le i \le j \le \lfloor k/2 \rfloor - 1$, then
 \begin{align}\label{strong-induced-2}
 &s_{(1^i)}(\x)s_{(1^j)}(\x)s_{(n-k+1,1^{k-i-1})}(\x)s_{(n-k+1,1^{k-j-1})}(\x)\nonumber \\
 &\qquad \qquad \qquad \ge_{s} s_{(1^{i-1})}(\x)s_{(1^{j+1})}(\x)s_{(n-k+1,1^{k-i})}(\x)s_{(n-k+1,1^{k-j-2})}(\x).
 \end{align}
\end{itemize}
\end{lemma}
\begin{proof}
$(1)$ 
We first treat the case $i=1$.
By the Pieri rule, the product $s_{(1)}(\x)s_{(n-k+1,1^{k-2})}(\x)$ is a sum of Schur functions obtained by adding a single box to the Young diagram of $(n-k+1,1^{k-2})$.
In particular, it contains the summand $s_{(n-k+1,1^{k-1})}(\x)$, so the difference
$s_{(1)}(\x)s_{(n-k+1,1^{k-2})}(\x)-s_{(n-k+1,1^{k-1})}(\x)$ is Schur positive.
This is precisely \eqref{strong-induced-1} when $i=1$.

Now assume $2\le i \le\lfloor k/2\rfloor$.

To prove \eqref{strong-induced-1}, we first apply \eqref{schur-posi-1} from Lemma \ref{lem-schur-log-1}
with
$(m,j,u,v)=(1,n-k+1,i-2,k-i)$, which gives
\begin{align}\label{strong-induced-1-pro1}
 s_{(1^{i-1})}(\x)s_{(n-k+1,1^{k-i})}(\x) \le_{s}
 s_{(1,1^{k-i-1})}(\x)s_{(n-k+1,1^{i-1})}(\x).
\end{align}
Applying \eqref{schur-posi-2} from Lemma \ref{lem-schur-log-1}
with
$(m,j,u,v)=(1,n-k+1,k-i-1,i-1)$, we have
\begin{align}\label{strong-induced-1-pro2}
 s_{(1,1^{k-i-1})}(\x)s_{(n-k+1,1^{i-1})}(\x) \le_{s} s_{(1^{i})}(\x)s_{(n-k+1,1^{k-i-1})}(\x).
\end{align}
Combining \eqref{strong-induced-1-pro1} and \eqref{strong-induced-1-pro2} yields \eqref{strong-induced-1}.

$(2)$
For \eqref{strong-induced-2}, let $1 \le i \le j \le \lfloor k/2 \rfloor - 1$.
If $i=1$, then the Pieri rule implies
\begin{align}\label{strong-induced-2-pro1-i1}
 s_{(1^{j+1})}(\x) \le_s s_{(1)}(\x)s_{(1^{j})}(\x).
\end{align}
Applying \eqref{schur-posi-1} from Lemma~\ref{lem-schur-log-1} with
$(m,j,u,v)=(n-k+1,n-k+1,k-j-2,k-1)$ gives
\begin{align}\label{strong-induced-2-pro2-i1}
 s_{(n-k+1,1^{k-1})}(\x)s_{(n-k+1,1^{k-j-2})}(\x)
 \le_s
 s_{(n-k+1,1^{k-2})}(\x)s_{(n-k+1,1^{k-j-1})}(\x).
\end{align}
Since products of Schur-positive symmetric functions are Schur positive, multiplying \eqref{strong-induced-2-pro1-i1} and \eqref{strong-induced-2-pro2-i1} yields \eqref{strong-induced-2} for $i=1$.

Now assume $2\le i \le j \le \lfloor k/2 \rfloor - 1$.

Applying \eqref{schur-posi-1} from Lemma \ref{lem-schur-log-1}
with
$(m,j,u,v)=(1,1,i-2,j)$ yields
\begin{align}\label{strong-induced-2-pro1}
 s_{(1^{i-1})}(\x)s_{(1^{j+1})}(\x) \le_{s} s_{(1^i)}(\x)s_{(1^j)}(\x).
\end{align}
Similarly, applying \eqref{schur-posi-1} from Lemma \ref{lem-schur-log-1}
with
$(m,j,u,v)=(n-k+1,n-k+1,k-j-2,k-i)$ gives
\begin{align}\label{strong-induced02-pro2}
 s_{(n-k+1,1^{k-i})}(\x)s_{(n-k+1,1^{k-j-2})}(\x) \le_s s_{(n-k+1,1^{k-i-1})}(\x)s_{(n-k+1,1^{k-j-1})}(\x).
\end{align}
The product of Schur-positive symmetric functions is again
Schur positive by the Littlewood--Richardson rule
\cite[Section~7.10]{stanley1999enumerative}.
Combining \eqref{strong-induced-2-pro1} with \eqref{strong-induced02-pro2}
implies \eqref{strong-induced-2}.
\end{proof}

With Lemma~\ref{lemma-strong-induced} established, we proceed to prove the strongly induced log-concavity of
$Y_{\U_{k,n}}^{\s_{n}}(t)$.

\begin{proof}[Proof of Theorem \ref{thm-invz-uni-unimodal-log}]
By \cite[Proposition~3.2]{gao2023induced}, a sequence
$\{C_i\}_{i\ge 0}^n$ of $\mathfrak{S}_n$-representations is strongly
inductively log-concave with respect to
$(\mathfrak{S}_{2n}, \mathfrak{S}_n \times \mathfrak{S}_n, \mathrm{id})$
if and only if
$
(\ch C_i)(\ch C_j)-(\ch C_{i-1})(\ch C_{j+1})
$
is Schur positive for all $1\le i\le j\leq n-1$.
Hence, it suffices to show that for any
$1\le i\le j\le k-1$,
\begin{equation*}
[t^i]\ch Y_{\U_{k,n}}^{\mathfrak{S}_n}(t)\,
[t^j]\ch Y_{\U_{k,n}}^{\mathfrak{S}_n}(t)
\;\ge_s\;
[t^{i-1}]\ch Y_{\U_{k,n}}^{\mathfrak{S}_n}(t)\,
[t^{j+1}]\ch Y_{\U_{k,n}}^{\mathfrak{S}_n}(t).
\end{equation*}
	Using the explicit expression
	\eqref{eq-ekl-uniform} from
	Theorem~\ref{thm-equi-uni}, we have
\begin{align*}
\ch Y_{\U_{k,n}}^{\mathfrak{S}_n}(t)
=
\sum_{i=0}^{\lfloor k/2\rfloor}
s_{(1^i)}(\x)\,
s_{(n-k+1,1^{k-i-1})}(\x)\, t^i
+
\sum_{i=0}^{\lfloor (k-1)/2\rfloor}
s_{(1^i)}(\x)\,
s_{(n-k+1,1^{k-i-1})}(\x)\, t^{k-i}.
\end{align*}
For notational convenience, set
\begin{align*}
a_{i}:=s_{(1^i)}(\x)
 \qquad \text{and} \qquad
 b_{i}:=s_{(n-k+1,1^{k-i-1})}(\x),
\end{align*}
for $0 \leq i \leq \lfloor k/2 \rfloor$.

Suppose that $k$ is odd.
Then the coefficient sequence of
$\ch Y_{\U_{k,n}}^{\mathfrak{S}_n}(t)$ is palindromic and given by
\begin{align*}
 a_{0}b_{0}, \, a_{1}b_{1}, \, \dots, \, a_{\frac{k-1}{2} }b_{\frac{k-1}{2}}, \, a_{\frac{k-1}{2}}b_{\frac{k-1}{2}}, \, \dots,a_{1}b_{1}, \, a_{0}b_{0}.
\end{align*}
By symmetry, it suffices to verify 
\begin{align}\label{k-odd-1}
 a_i b_i\,a_j b_j
 \,\ge_s\,
 a_{i-1} b_{i-1}\, a_{j+1} b_{j+1},
 \qquad \text{~for~}
 1 \le i \le j \le (k-3)/2,
\end{align}
\begin{align}\label{k-odd-2}
 a_i b_i\,a_j b_j
 \,\ge_s\,
 a_{i-1} b_{i-1}\, a_{j-1} b_{j-1},
 \qquad \text{~for~}
 1 \le i,j \le (k-1)/2,
\end{align}
and
\begin{align}\label{k-odd-3}
 a_{i}b_{i}\ge_{s} a_{i-1}b_{i-1},
 \qquad \text{~for~}
 1 \le i \le (k-3)/2.
\end{align}
The inequalities \eqref{k-odd-1} follow directly from
\eqref{strong-induced-2} in
Lemma~\ref{lemma-strong-induced}.
Moreover, \eqref{strong-induced-1} implies that
$a_i b_i \ge_s a_{i-1} b_{i-1}$ for all
$1\le i\le (k-1)/2$, which establishes
\eqref{k-odd-2} and \eqref{k-odd-3}.

Now suppose that $k$ is even.
In this case, the coefficient sequence of $\ch \, Y_{\U_{k,n}}^{\s_n}(t)$ is
\begin{align*}
 a_{0}b_{0}, \, a_{1}b_{1}, \, \dots, \, a_{\frac{k}{2}-1 }b_{\frac{k}{2}-1},
 \, a_{\frac{k}{2} }b_{\frac{k}{2}},\, a_{\frac{k}{2}-1}b_{\frac{k}{2}-1}, \, \dots,a_{1}b_{1}, \, a_{0}b_{0}.
\end{align*}
Again by symmetry, it suffices to verify
\begin{align}\label{k-even-1}
 a_i b_i\,a_{j} b_{j}
 \,\ge_s\,
 a_{i-1} b_{i-1}\, a_{j+1} b_{j+1},
 \qquad \text{~for~}
 1 \le i \le j \le k/2-1,
\end{align}
and
\begin{align}\label{k-even-2}
 a_i b_i\,a_{j} b_{j}
 \,\ge_s\,
 a_{i-1} b_{i-1}\, a_{j-1} b_{j-1},
 \qquad \text{~for~}
 1 \le i,j \le k/2.
\end{align}
The inequality \eqref{k-even-1} follows from
\eqref{strong-induced-2}, while
\eqref{k-even-2} is a direct consequence of
\eqref{strong-induced-1}.
This completes the proof.
\end{proof}

\subsection{\texorpdfstring{$Q$}{Q}-niform matroids}\label{Q-niform}

In this subsection, we derive an explicit formula for the equivariant inverse 
$Z$-polynomial of the $q$-niform matroid under the action of 
general linear group, and we prove that its coefficients 
are both equivariantly unimodal and strongly inductively log-concave.
Throughout, we fix a prime power $q$ and denote by $\mathbb{F}_q$ the finite field with $q$ elements.

A $q$-analogue of the uniform matroid $\U_{k,n}$, denoted by $\U_{k,n}(q)$, was introduced in \cite{hameister2021chow,proudfoot2019equivariant}. 
Following Proudfoot~\cite{proudfoot2019equivariant}, we refer to $\U_{k,n}(q)$ as the $q$-niform matroid. 
Its ground set consists of all hyperplanes of $V=\mathbb{F}_q^n$, and a basis is a set of $k$ hyperplanes whose intersection has dimension $n-k$. 
Hameister, Rao, and Simpson \cite{hameister2021chow} defined 
$\U_{k,n}(q)$ as the matroid on $V$ whose bases are the linearly 
independent $k$-subsets of $V$. This matroid is not simple, since 
any two nonzero scalar multiples of the same vector form parallel elements. 
Proudfoot's definition coincides with theirs after identifying each 
parallel class with a single element.
Precisely, for each hyperplane $H$, choose a nonzero linear functional $\alpha_H\in V^*$ such that $\ker(\alpha_H)=H$, then $\dim\bigl(\bigcap_{i=1}^k H_i\bigr)=n-k$ if and only if the corresponding functionals $\alpha_{H_1},\dots,\alpha_{H_k}$ are linearly independent. Hence, Proudfoot's definition for $\U_{k,n}(q)$ coincides with the matroid on the projective space $P(V^*)$
considered by
Hameister, Rao, and Simpson.

The general linear group
$\mathrm{GL}_n(\mathbb{F}_q)$ acts naturally on $V$, and hence on the set
of hyperplanes. This induces an action of
$\mathrm{GL}_n(\mathbb{F}_q)$ on the $q$-niform matroid $\U_{k,n}(q)$.
In order to compute the equivariant inverse $Z$-polynomial of $\U_{k,n}(q)$, we recall some background on unipotent representations of
finite general linear groups; see \cite{proudfoot2019equivariant}.
The irreducible unipotent representations
of $\mathrm{GL}_n(\mathbb{F}_q)$ are indexed by partitions of $n$.
For a partition $\lambda \vdash n$, let $V_\lambda(q)$ denote the
corresponding irreducible unipotent representation. These representations
may be regarded as a $q$-analogue of the Specht modules $V_\lambda$ of
$\mathfrak{S}_n$. 
If a representation is isomorphic to a direct sum of
irreducible unipotent representations, then it is called a unipotent representation.

The connection between the uniform and $q$-niform settings is provided by the
Comparison Theorem; see \cite[Theorem~B]{curtis1975reduction} and
\cite[Theorem~2.1]{proudfoot2019equivariant}. 
For $0 \le i \le n$, let
$P_{i,n}(\mathbb{F}_q) \subseteq \mathrm{GL}_n(\mathbb{F}_q)$
denote the standard parabolic subgroup with Levi factor
$\mathrm{GL}_i(\mathbb{F}_q)\times \mathrm{GL}_{n-i}(\mathbb{F}_q)$,
and let
$
\iota \colon
P_{i,n}(\mathbb{F}_q) \rightarrow
\mathrm{GL}_i(\mathbb{F}_q)\times \mathrm{GL}_{n-i}(\mathbb{F}_q)
$
be the natural projection.
Given representations $A$ of $\mathrm{GL}_i(\mathbb{F}_q)$ and
$B$ of $\mathrm{GL}_{n-i}(\mathbb{F}_q)$, we write
$$
A * B :=
\Ind_{P_{i,n}(\mathbb{F}_q)}^{\mathrm{GL}_n(\mathbb{F}_q)} (A \otimes B),
$$
which is Harish--Chandra induction.
Let $k\leq n$ be a pair of natural numbers and let $\lambda$, $\mu$ and $\nu$ be partitions of $n$, $k$ and $n-k$ respectively.
The Comparison Theorem states that the multiplicity of $V_\lambda(q)$ in $V_\mu(q)*V_{\nu}(q)$ is equal to the multiplicity of $V_{\lambda}$ in the induced representation $\mathrm{Ind}(V_\mu \otimes V_\nu)$.

\begin{lemma}\label{lem-mobius-equi-qniform}
For any equivariant $q$-niform matroid $\mathrm{GL}_n(\mathbb{F}_q)\curvearrowright\U_{k,n}(q)$ with $n\geq k \geq 1$, the equivariant M\"obius invariant is given by
\begin{align}\label{equi-qniform-mobius}
 \mu_{\U_{k,n}(q)}^{\mathrm{GL}_n(\mathbb{F}_q)}
 = (-1)^k V_{(n-k+1,1^{k-1})}(q).
\end{align}
\end{lemma}

\begin{proof}
By the Comparison Theorem,
Proudfoot~\cite{proudfoot2019equivariant} showed that the equivariant
characteristic polynomial of
$\mathrm{GL}_{n}(\mathbb{F}_q)\curvearrowright \U_{k,n}(q)$ is
\begin{equation}\label{eq:proudfoot-char}
H_{\U_{k,n}(q)}^{\mathrm{GL}_{n}(\mathbb{F}_q)}(t)
=
\sum_{i=0}^{k-1}
(-1)^i
\bigl(
V_{(n-i,1^i)}(q)
+
V_{(n-i+1,1^{i-1})}(q)
\bigr)
t^{\,k-i}
+
(-1)^k V_{(n-k+1,1^{k-1})}(q).
\end{equation}
Recall that the equivariant M\"obius invariant
is the constant term of the equivariant characteristic polynomial.
Taking the constant term of \eqref{eq:proudfoot-char} yields
\eqref{equi-qniform-mobius}.
\end{proof}

The following theorem gives a decomposition of the coefficients of the equivariant inverse $Z$-polynomial of
$\mathrm{GL}_n(\mathbb{F}_{q})\curvearrowright \U_{k,n}(q)$ into irreducible unipotent representations.

\begin{theorem}\label{prop-q-niform-equi}
For any equivariant $q$-niform matroid
$\mathrm{GL}_n(\mathbb{F}_q)\curvearrowright \U_{k,n}(q)$ with $n \ge k \ge 1$, we have
\begin{align}\label{q-niform-equi}
Y_{\U_{k,n}(q)}^{\mathrm{GL}_n(\mathbb{F}_q)}(t)
 = & \sum_{i=0}^{\lfloor k/2 \rfloor}
 \sum_{x=0}^{i}
 \Big(
 V_{(n-k+1,2^{x},1^{k-2x-1})}(q)
 + V_{(n-k+2,2^{x-1},1^{k-2x})}(q)
 \Big) t^{i} \nonumber\\
 &\quad
 + \sum_{i=0}^{\lfloor (k-1)/2 \rfloor}
 \sum_{x=0}^{i}
 \Big(
 V_{(n-k+1,2^{x},1^{k-2x-1})}(q)
 + V_{(n-k+2,2^{x-1},1^{k-2x})}(q)
 \Big) t^{k-i},
\end{align}
where $V_\lambda(q)=0$ whenever $\lambda$ is not a partition.
\end{theorem}

\begin{proof}
As in the proof of Theorem~\ref{thm-equi-uni},
the $\mathrm{GL}_n(\mathbb{F}_q)$-orbits of flats of $\U_{k,n}(q)$ are indexed by
their sizes $0 \le i \le k$.
The stabilizer of a flat of size $i$ is the parabolic subgroup
$P_{i,n}(\mathbb{F}_q)$, and the corresponding localizations and
contractions are again $q$-niform matroids.
Applying definition~\eqref{YM}, we obtain an expression for
$Y_{\U_{k,n}(q)}^{\mathrm{GL}_n(\mathbb{F}_q)}(t)$ in terms of
Harish--Chandra induction.
More precisely, we have
\begin{align}\label{eq-q-ekl-uniform}
Y_{\U_{k,n}(q)}^{\mathrm{GL}_n(\mathbb{F}_q)}(t)
 = & \sum_{i=0}^{\lfloor k/2 \rfloor}
 \Big(
 V_{(1^i)}(q) * V_{(n-k+1,1^{k-i-1})}(q)
 \Big) t^{i} \nonumber\\
 &\qquad
 + \sum_{i=0}^{\lfloor (k-1)/2 \rfloor}
 \Big(
 V_{(1^i)}(q) * V_{(n-k+1,1^{k-i-1})}(q)
 \Big) t^{k-i}.
\end{align}
Here $V_{(1^i)}(q)$ arises from the equivariant inverse
Kazhdan--Lusztig polynomial of $\U_{n,n}(q)$; see 
\cite[Theorem~4.7]{gao2023induced}.
On the other hand, $V_{(n-k+1,1^{k-i-1})}(q)$ is determined by the equivariant
M\"obius invariant of $\U_{k,n}(q)$; see 
Lemma~\ref{lem-mobius-equi-qniform}.
Since Proposition~\ref{prop-equi-uni} gives the irreducible decomposition of the coefficients of
$Y_{\U_{k,n}}^{\mathfrak{S}_n}(t)$, the Comparison Theorem implies that
the same multiplicities arise in the $q$-niform setting.
Thus we complete the proof.
\end{proof}

We now consider the equivariant unimodality and strongly
induced log-concavity of the equivariant inverse
$Z$-polynomials for $q$-niform matroids.
For small ranks $1 \le k \le 5$, Gao et al.~\cite[Theorem~4.1]{gao2023induced}
obtained the strongly induced log-concavity for the equivariant $Z$-polynomials
of $\mathrm{GL}_n(\mathbb{F}_q)\curvearrowright \U_{k,n}(q)$.
For a sequence of virtual unipotent representations
of $\mathrm{GL}_n(\mathbb{F}_q)$, strongly induced log-concavity
is defined with respect to the triple
$
(G,H,\phi)
=
(\mathrm{GL}_{2n}(\mathbb{F}_q),P_{n,2n}(\mathbb{F}_q), \iota).
$
By the Comparison Theorem, this notion is equivalent to strongly
induced log-concavity for the corresponding sequence of virtual
representations of $\mathfrak{S}_n$.
The same equivalence holds for
equivariant unimodality.
Hence, 
we obtain the following result for the equivariant inverse $Z$-polynomials
of $\mathrm{GL}_n(\mathbb{F}_q)\curvearrowright\U_{k,n}(q)$.

\begin{theorem}\label{equi-q-ni-log}
Let $n\ge k\ge 1$. The equivariant inverse $Z$-polynomial $Y_{\U_{k,n}(q)}^{\mathrm{GL}_n(\mathbb{F}_q)}(t)$ is equivariant unimodal and strongly inductively log-concave.
\end{theorem}

Let
$[m]_q := 1 + q + \cdots + q^{m-1}$ and $
[n]_q! := \prod_{i=1}^n [i]_q.
$
Recall that the Gaussian binomial coefficient is defined by
\begin{equation*}
\begin{bmatrix} n \\ k \end{bmatrix}_q
:=
\frac{[n]_q!}{[k]_q!\,[n-k]_q!}.
\end{equation*}
By taking dimensions of the unipotent representations appearing in \eqref{eq-q-ekl-uniform}, we obtain a closed formula for the ordinary inverse $Z$-polynomial of $q$-niform matroid.

\begin{proposition}
For any $q$-niform matroid $\U_{k,n}(q)$ with $n\geq k \geq 1$, we have
\begin{align*}
Y_{\U_{k,n}(q)}(t) = & \sum_{i=0}^{\lfloor k/2 \rfloor}
 q^{\binom{i}{2}+\binom{k-i}{2}}\begin{bmatrix} n \\ i\end{bmatrix}_q \begin{bmatrix} n-i-1 \\ k-i-1\end{bmatrix}_q
 t^{i} \nonumber\\
 &\qquad
 + \sum_{i=0}^{\lfloor (k-1)/2 \rfloor}
 q^{\binom{i}{2}+\binom{k-i}{2}}\begin{bmatrix} n \\ i\end{bmatrix}_q \begin{bmatrix} n-i-1 \\ k-i-1\end{bmatrix}_q
 t^{k-i}.
\end{align*}
\end{proposition}

\begin{proof}
Taking dimensions in \eqref{eq-q-ekl-uniform} yields
\begin{align}\label{or-eq-q-ekl-uniform}
Y_{\U_{k,n}(q)}(t)
 = & \sum_{i=0}^{\lfloor k/2 \rfloor}
 \Big(
 \dim V_{(1^i)}(q) * V_{(n-k+1,1^{k-i-1})}(q)
 \Big) t^{i} \nonumber\\
 &\qquad
 + \sum_{i=0}^{\lfloor (k-1)/2 \rfloor}
 \Big(
 \dim V_{(1^i)}(q) * V_{(n-k+1,1^{k-i-1})}(q)
 \Big) t^{k-i}.
\end{align}
By
\cite[Theorem~2.1]{proudfoot2019equivariant}, we have
$$\dim V_{(1^i)}(q)=q^{\binom{i}{2}}
\quad \text{~and~} \quad
\dim V_{(n-k+1,1^{k-i-1})}(q)=q^{\binom{k-i}{2}}
\begin{bmatrix} n-i-1 \\ {k-i-1} \end{bmatrix}_q.
$$
Since $P_{i,n}(\mathbb{F}_q)$ is the stabilizer of an $i$-dimensional subspace of $\mathbb{F}_q^n$, it follows that
$$
|\mathrm{GL}_n(\mathbb{F}_q):P_{i,n}(\mathbb{F}_q)|
=
\begin{bmatrix} n \\ i\end{bmatrix}_q
.
$$
Substituting these expressions into \eqref{or-eq-q-ekl-uniform} completes
the proof.
\end{proof}

Applying the dimension map yields the strong log-concavity of the coefficients of
$Y_{\U_{k,n}(q)}(t)$.

\begin{corollary}\label{cor-qniform-ordinary-logconcave}
Let $n\ge k\ge 1$ be integers. The inverse $Z$-polynomial $Y_{\U_{k,n}(q)}(t)$ is strongly log-concave.
\end{corollary}

\section{Paving matroids}\label{subse-paving}

The aim of this section is to prove
Theorem~\ref{thm-paving}.
Based on Proposition~\ref{prop-equi-uni}, we derive an explicit
relation between the equivariant inverse $Z$–polynomials of 
matroid $W \curvearrowright \M$ and those of its relaxation
along a $W$–orbit of stressed hyperplanes.
This approach allows us to compute the equivariant inverse
$Z$–polynomials of arbitrary equivariant paving matroids.

Let $\M$ be a matroid of rank $k$ equipped with an action of a finite group $W$, and let $H$ be a stressed hyperplane of $\M$ with stabilizer $W_H$. Write $\mathcal{H}$ for the $W$-orbit of $H$, and define
\begin{align*}
 \mathcal{A} = \{ A \subseteq H_i : |A| = k-1,\, H_i \in \mathcal{H} \}, \qquad \qquad
 \mathcal{C} = \{ A \subseteq H : |A| = k-1 \}.
\end{align*}
Clearly, $\mathcal{C} \subseteq \mathcal{A} $.
The following lemma shows that restricting the $W$–orbit decomposition of
$\mathcal A$ to $\mathcal C$ yields exactly the $W_H$–orbit decomposition of
$\mathcal C$. This result will be used in the proof of
Theorem~\ref{thm-paving}.

\begin{lemma}\label{lem-paving-1}
The partition of $\mathcal A$ into $W$-orbits, when restricted to
$\mathcal C$,
coincides with the partition of $\mathcal C$ into $W_H$-orbits.
In particular,
$$
|\mathcal A / W| = |\mathcal C / W_H|,
$$
and every $W$-orbit in $\mathcal A$ contains a unique $W_H$-orbit in $\mathcal C$.
\end{lemma}

\begin{proof}
We first show that every $W$-orbit in $\mathcal A$ intersects $\mathcal C$.
Let $A \in \mathcal A$.
Then $A \subset H_i$ for some hyperplane $H_i \in \mathcal H$.
Since $W$ acts transitively on hyperplanes, there exists $w \in W$
such that $wH_i = H$.
Consequently, $wA \subset H$, so $wA \in \mathcal C$.
This proves that each $W$--orbit in $\mathcal A$ contains at least one element
of $\mathcal C$.

We proceed to compare the orbit partitions on $\mathcal C$.
Since $W_H \subseteq W$, every $W_H$-orbit in $\mathcal C$
is contained in a single $W$-orbit.
It remains to show that two elements of $\mathcal C$
lying in the same $W$--orbit must already lie in the same $W_H$-orbit.
Let $A_1, A_2 \in \mathcal C$ with $|A_1|=|A_2|=k-1$, and suppose
$A_2 = wA_1$ for some $w \in W$.
Since $A_2 \subset H$ and $A_2 \subset wH$, we have
$$
|H \cap wH| \ge k-1.
$$
By \cite[Proposition~3.8]{ferroni2023stressed}, two distinct hyperplanes
intersect in at most $k-2$ elements.
Therefore $wH = H$, which implies $w \in W_H$.
Hence $A_1$ and $A_2$ lie in the same $W_H$-orbit.

Combining the above arguments, we conclude that the restriction of the
$W$--orbit partition of $\mathcal A$ to $\mathcal C$ coincides exactly with
the $W_H$--orbit partition of $\mathcal C$.
In particular, the two partitions have the same number of equivalence classes,
and representatives for $W$-orbits in $\mathcal A$ may be chosen directly
from $\mathcal C$.
\end{proof}

With Lemma~\ref{lem-paving-1} established, we now analyze the effect of relaxing the entire $W$-orbit $\mathcal{H}$ of the stressed hyperplane~$H$.
By \cite[Proposition~3.9]{ferroni2023stressed}, relaxing one stressed hyperplane keeps other stressed hyperplanes stressed, and hence relaxing any collection of stressed hyperplanes is independent of the order in which they are relaxed.
Let $\widetilde{\M}$ denote the matroid obtained by relaxing all hyperplanes in $\mathcal{H}$ simultaneously.
Its bases are those of $\M$ together with all $k$-subsets contained in a unique element of $\mathcal{H}$.
Since this relaxation is $W$-equivariant, the $W$-action on $\M$ extends naturally to $\widetilde{\M}$.
The following result describes how this operation modifies the equivariant inverse $Z$–polynomial.

\begin{theorem}\label{pro-relaxM-M}
Fix integers $h\ge k\ge 1$.
Let $\M$ be a matroid of rank $k$ equipped with a $W$-action, and let $H$ be a stressed hyperplane of $\M$ of size $h$.
Let $\widetilde{\M}$ be the matroid obtained from $\M$ by relaxing every hyperplane in the $W$-orbit of $H$.
Then there exists a polynomial $y^{\mathfrak S_h}_{k,h}(t)$ with coefficients in the virtual representation ring of $\mathfrak S_h$ such that
\begin{align*}
Y^{W}_{\widetilde{\M}}(t)-Y^{W}_{\M}(t)=\mathrm{Ind}^{W}_{W_H}\mathrm{Res}^{\s_h}_{W_H}y^{\s _h}_{k,h}(t),
\end{align*}
where $W_H\subseteq W$ denotes the stabilizer of $H$.
\end{theorem}

\begin{proof}
Set
\begin{align*}
 R^{W}_{\M}(t):=Y^{W}_{\M}(t)-Q^{W}_{\M}(t).
\end{align*}
Then
\begin{align*}
 R^{W}_{\widetilde{\M}}(t)-R^{W}_{\M}(t)= \Big(Y^{W}_{\widetilde{\M}}(t)-Y^{W}_{\M}(t)\Big)
 -
 \Big(Q^{W}_{\widetilde{\M}}(t)-Q^{W}_{\M}(t)\Big).
\end{align*}
Note that $Q^{W}_{\M}(t)$ has degree strictly less than $\tfrac{k}{2}$, while
$Y^{W}_{\M}(t)$ is a polynomial of degree $k$ whose coefficients satisfy
$t^{k}Y^{W}_{\M}(t^{-1}) = Y^{W}_{\M}(t).$
It follows that the difference
$Y^{W}_{\widetilde{\M}}(t)-Y^{W}_{\M}(t)$ is again a
polynomial of degree at most $k$ whose coefficients are symmetric, while
$Q^{W}_{\widetilde{\M}}(t)-Q^{W}_{\M}(t)$ has degree strictly less than $\frac{k}{2}$.
Since a polynomial of degree $<\tfrac{k}{2}$ cannot contribute to
the coefficient-symmetric part in degrees $\geq \tfrac{k}{2}$, it suffices to prove that
there exists a polynomial
$r^{\mathfrak S_h}_{k,h}(t)$ with coefficients in the virtual representation
ring of $\mathfrak S_h$ such that
\begin{align*}
 R^{W}_{\widetilde{\M}}(t) - R^{W}_{\M}(t)
 =
 \mathrm{Ind}^{W}_{W_H}
 \mathrm{Res}^{\mathfrak S_h}_{W_H}
 r^{\mathfrak S_h}_{k,h}(t).
\end{align*}

Recall that $\mathcal{H}$ denotes the $W$-orbit of $H$ and $ \mathcal{A} = \{ A \subseteq H_i : |A| = k-1, \, H_i \in \mathcal{H} \}$. By \cite[Proposition 3.10]{ferroni2023stressed}, we have
\begin{align}\label{flat-1}
 \mathcal{L}(\widetilde{\M}) = (\mathcal{L}(\M) \smallsetminus \mathcal{H}) \sqcup \mathcal{A}.
\end{align}
Since
\begin{align}\label{r-eq-2}
 R^{W}_{\M}(t)=(-1)^{\rk(\M)}\sum_{[F] \in \left(\mathcal{L}(\M) \smallsetminus \{E\} \right)/W}(-1)^{\rk(F)}t^{\rk(\M/F)}~\mathrm{Ind}^{W}_{W_F}
 Q^{W_F}_{\M|_{F}}(t) \otimes \mu^{W_F}_{\M/F},
\end{align}
expanding $R^{W}_{\widetilde{\M}}(t)-R^{W}_{\M}(t)$
according to \eqref{flat-1} and \eqref{r-eq-2} yields three types of contributions.
We now analyze these three contributions separately and show that
each of them can be expressed in the form
$
\Ind^{W}_{W_H}
\Res^{\mathfrak{S}_h}_{W_H}(\cdot).
$
as required.

\emph{(1) The contribution coming from the flat $H$ itself.}
This gives the term
\begin{align}\label{threepart-1}
t^{\rk(\M/H)}~\mathrm{Ind}^{W}_{W_H}
Q^{W_H}_{\M|_{H}}(t) \otimes \mu^{W_H}_{\M/H}.
\end{align}
The isomorphism $\M|_H \cong \U_{k-1,h}$ yields
\begin{align*}
 Q^{W_H}_{\M|_{H}}(t) = Q^{W_H}_{\U_{k-1,h}}(t) = \mathrm{Res}^{\s_h}_{W_H}Q^{\s_h}_{\U_{k-1,h}}(t).
\end{align*}
Moreover, $\M/H\cong\U_{1,|E|-h}$ and
Lemma~\ref{lem-mobius-equi-uni} imply
\begin{align*}
 \mu^{W_H}_{\M/H} = \mu^{W_H}_{\U_{1,|E|-h}} = \mathrm{Res}^{\s_{|E|-h}}_{W_H}\mu^{\s_{|E|-h}}_{\U_{1,|E|-h}} = -\tau_{W_H}. 
\end{align*}
Hence, the term \eqref{threepart-1} turns out to be
\begin{align*}
 t~\mathrm{Ind}^{W}_{W_H}\Big(
\mathrm{Res}^{\s_h}_{W_H}Q^{\s_h}_{\U_{k-1,h}}(t) \otimes \big(-\tau_{W_H}\big)\Big)
 =& -t~\mathrm{Ind}^{W}_{W_H}
 \mathrm{Res}^{\s_h}_{W_H}Q^{\s_h}_{\U_{k-1,h}}(t).
\end{align*}
Since $Q^{\mathfrak{S}_h}_{\U_{k-1,h}}(t)$ depends only on $k$ and $h$,
this contribution is indeed of the asserted form.


\emph{(2) The contribution coming from the $W$-orbits in $\mathcal{A}$.}
It takes the form
\begin{align}\label{threepert-2}
 \sum_{[F] \in \mathcal{A}/W}(-1)^{k+\rk_{\widetilde{\M}}(F)}t^{\rk(\widetilde{\M}/F)}~\mathrm{Ind}^{W}_{W_F}
Q^{W_F}_{\widetilde{\M}|_{F}}(t) \otimes \mu^{W_F}_{\widetilde{\M}/F}.
\end{align}
For each $F \in \mathcal{A}$, the restriction satisfies $\widetilde{\M}|_F \cong B_{k-1}$, which gives
\begin{align*}
 Q^{W_F}_{\widetilde{\M}|_{F}}(t) = Q^{W_F}_{B_{k-1}}(t) = \mathrm{Res}^{\s_{k-1}}_{W_F} Q^{\s_{k-1}}_{B_{k-1}}(t)=\mathrm{Res}^{\s_{k-1}}_{W_F} V_{(1^{k-1})}.
\end{align*}
where the last equality follows from \eqref{boo-uni-invkl1}.
The contraction satisfies $\widetilde{\M}/F \cong \U_{1,|E|-k+1}$, and
Lemma~\ref{lem-mobius-equi-uni} gives
\begin{align*}
 \mu^{W_F}_{\widetilde{\M}/F} = \mu^{W_F}_{\U_{1,|E|-k+1}} = \mathrm{Res}^{\s_{|E|-k+1}}_{W_F} \mu^{\s_{|E|-k+1}}_{\U_{1,|E|-k+1}} = - \tau_{W_F}.
\end{align*}
By Lemma~\ref{lem-paving-1}, summation over $[F]\in\mathcal{A}/W$
may equivalently be taken over $[F]\in\mathcal{C}/W_H$.
Moreover, for each $F\in\mathcal{C}$, the stabilizer satisfies $W_F\subseteq W_H$. In fact,
if $w\in W_F$ then $w$ stabilizes $F$ setwise, so
$|wH\cap H|\ge|F|=k-1$; but \cite[Proposition~3.8]{ferroni2023stressed}
states that two distinct hyperplanes intersect in size at most $k-2$,
forcing $wH=H$ and hence $w\in W_H$.
Combining this with the fact that $\rk_{\widetilde{\M}}(F)=k-1$,
the sum~\eqref{threepert-2} equals
\begin{align*}
&-t\sum_{[F] \in \mathcal{C} /W_H }\mathrm{Ind}^{W}_{W_H} \mathrm{Ind}^{W_H}_{W_F}
\Big (\mathrm{Res}^{\s_{k-1}}_{W_F} V_{(1^{k-1})}
\otimes \big(- \tau_{W_F}\big)\Big) \notag
\\
 = & t \sum_{[F] \in \mathcal{C} /W_H }\mathrm{Ind}^{W}_{W_H} \mathrm{Ind}^{W_H}_{W_F}
 \mathrm{Res}^{\s_{k-1}}_{W_F} V_{(1^{k-1})} \notag \\
 = &
t~\mathrm{Ind}^{W}_{W_H}\sum_{[F] \in \mathcal{C}/W_H}\mathrm{Ind}^{W_H}_{W_F}
\mathrm{Res}^{\s_{k-1}}_{W_F} V_{(1^{k-1})}. 
\end{align*}
Since $W_H$ acts on $\mathcal{C}$ by permutation,
\cite[Lemma~2.7]{proudfoot2021equivariant}
allows us to replace the sum over $W_H$–orbits
by the sum over all elements of $\mathcal{C}$.
Consequently, we obtain
\begin{align*}
t~\mathrm{Ind}^{W}_{W_H} \sum_{F \in \mathcal{C} }
\mathrm{Res}^{\s_{k-1}}_{W_F } V_{(1^{k-1})}
=
t~\mathrm{Ind}^{W}_{W_H} ~\mathrm{Res}^{\s_h}_{W_H} \Big(\sum_{F \in \mathcal{C} } V_{(1^{k-1})}\Big).
\end{align*}
Since each summand is the restriction of the same
$\mathfrak{S}_{k-1}$–representation,
the $\mathfrak{S}_h$–representation
$\sum_{F\in\mathcal{C}} V_{(1^{k-1})}$
can be identified via the standard description of permutation
induction with
$$
\Ind^{\mathfrak{S}_h}_{\mathfrak{S}_{k-1}\times\mathfrak{S}_{h-k+1}}
\bigl(
V_{(1^{k-1})}\otimes\tau_{\mathfrak{S}_{h-k+1}}
\bigr).
$$
Therefore,
\begin{align*} t~\mathrm{Ind}^{W}_{W_H} ~\mathrm{Res}^{\s_h}_{W_H} \Big(\Ind^{\s_h}_{\s_{k-1} \times \s_{h-k+1}}\big( V_{(1^{k-1})} \otimes \tau_{\s_{h-k+1}} \big)\Big),
\end{align*}
which is of the required form.


\emph{(3) The contribution from the remaining flats of $\M$,
that is, from $F\in\mathcal{L}(\M)\setminus(\mathcal{H}\cup\{E\})$.}
For all such $F$, we have $\rk_{\M}(F)=\rk_{\widetilde{\M}}(F)$, and hence this contribution is given by
\begin{align}\label{paving-term3}
\sum_{[F]}(-1)^{k+\rk_{\M}(F)}~\mathrm{Ind}^{W}_{W_F} \Big(t^{\rk(\widetilde{\M}/F)}
 Q^{W_F}_{\widetilde{\M}|_{F}}(t)\otimes \mu^{W_F}_{\widetilde{\M}/F} - t^{\rk(\M/F)}
 Q^{W_F}_{\M|_{F}}(t) \otimes \mu^{W_F}_{\M/F}\Big),
 \end{align}
where the sum ranges over
$(\mathcal{L}(\M)\setminus(\mathcal{H}\cup\{E\}))/W$.
Since $\widetilde{\M}|_F=\M|_F$ and $\rk(\widetilde{\M}/F)=\rk(\M/F)$ for all flats in the above index set, the two terms inside the parentheses differ only through the M\"{o}bius invariant.
Moreover, if $F$ is not contained in any hyperplane $H_i\in\mathcal{H}$, then $\widetilde{\M}/F=\M/F$, and the corresponding summand vanishes.
It follows that only flats $F\subset H_i\in\mathcal{H}$ with $|F|\le k-2$ contribute to \eqref{paving-term3}. Thus the sum reduces to
\begin{align}\label{paving-term3-3}
 \sum_{[F]} (-1)^{k+\rk_{\M}(F)}t^{\rk(\M/F)}~\mathrm{Ind}^{W}_{W_F}\Big( Q^{W_F}_{\M|_{F}}(t)\otimes \big(\mu^{W_F}_{\widetilde{\M}/F} - \mu^{W_F}_{\M/F}\big) \Big),
\end{align}
where the sum is taken over
$
\{F\mid F\subset H_i,\ |F|\le k-2,\ H_i\in\mathcal{H}\}/W.
$
For such $F$, the restriction $\M|_{F}$ is the Boolean matroid on $F$.

To analyze the difference
$\mu^{W_F}_{\widetilde{\M}/F}-\mu^{W_F}_{\M/F}$,
we first observe that the matroid $\widetilde{\M}/F$ admits the following
equivalent description.
It can be obtained from $\M$ by first relaxing all stressed hyperplanes
in $\mathcal{H}$ and then contracting $F$.
Equivalently, one may first contract $F$ and then relax $J\setminus F$ of those hyperplanes $J\in\mathcal{H}$ that contain $F$.
Indeed, if a stressed hyperplane $J\in\mathcal{H}$ does not contain $F$, then it becomes invisible in the contraction $\M/F$, so relaxing $J$ does not affect $\M/F$.
Consequently, only hyperplanes $J\in\mathcal{H}$ satisfying
$F\subset J$ contribute to the difference between $\M/F$ and
$\widetilde{\M}/F$.
We therefore define
$$
D(F,H):=\{J\in\mathcal{H}\mid F\subset J\}.
$$
By \cite[Corollary~9.8]{elias2025categorical}, we have
\begin{align}\label{paving-third-2} OS^{W}_{\widetilde{\M},k}-OS^{W}_{\M,k}= \Ind^{W}_{W_H}~\mathrm{Res}^{\s_h}_{W_H} \wedge^{k} V_{(h-1,1)}.
\end{align}
Combining \eqref{char-uni-3} with \eqref{paving-third-2} yields
\begin{align}\label{mobius-difference}
\mu^{W_F}_{\widetilde{\M}/F} - \mu^{W_F}_{\M/F} = (-1)^{k-|F|}\sum_{[J] \in D(F,H)/W_F} \Ind^{W_F}_{W_{J} \cap W_{F}} \mathrm{Res}^{\s_{h-|F|}}_{W_{J} \cap W_{F} } \wedge^{k-|F|} V_{(h-|F|-1,1)}.
\end{align}
Substituting \eqref{mobius-difference} into
\eqref{paving-term3-3} and using the standard identity \cite[Section~3.3]{serre1977linear}
$$
V\otimes \Ind_{H}^{G} V_1 \;=\;
\Ind_{H}^{G}\!\left(\mathrm{Res}_{H}^{G} V \otimes V_1\right)
$$
 for any subgroup $H\le G$, we obtain
 {\small
\begin{align}
&
\Ind^{W}_{W_F} \Big(
Q^{W_F}_{\M|_{F}}(t) \otimes (\mu^{W_F}_{\widetilde{\M}/F} - \mu^{W_F}_{\M/F})\Big)
\notag \\
 =&\Ind^{W}_{W_F}(-1)^{k-|F|} \Big(Q^{W_F}_{\M|_{F}}(t) \otimes \sum_{[J] \in D(F,H)/W_F} \mathrm{Ind}^{W_F}_{W_J \cap W_F} \mathrm{Res}^{\s_{h-|F|}}_{W_{J}\cap W_F} \wedge^{k-|F|} V_{(h-|F|-1,1)} \Big) \notag\\
 = &(-1)^{k-|F|} \mathrm{Ind}^{W}_{W_F}\sum_{[J] \in D(F,H)/W_F} \mathrm{Ind}^{W_F}_{W_J \cap W_F} \Big( \mathrm{Res}^{W_F}_{W_J \cap W_F} Q^{W_F}_{\M|_{F}}(t) \otimes \mathrm{Res}^{\s_{h-|F|}}_{W_{J}\cap W_F} \wedge^{k-|F|} V_{(h-|F|-1,1)} \Big). \label{paving-term3-expanded}
\end{align}}
For each such flat $F$, the restriction $\M|_F$ is a Boolean matroid on
$|F|$ elements.
By \eqref{boo-uni-invkl1}, we therefore have
\begin{align*}
Q^{W_F}_{\M|_{F}}(t) = \mathrm{Res}^{\s_{|F|}}_{W_F} Q^{\s_{|F|}}_{B_{|F|}}(t) = \mathrm{Res}^{\s_{|F|}}_{W_F} V_{(1^{|F|})}.
\end{align*}
Using the associativity of restriction, \eqref{paving-term3-expanded} can be rewritten as
\begin{align}
 &(-1)^{k-|F|}\mathrm{Ind}^{W}_{W_F}\sum_{[J] \in D(F,H)/W_F} \mathrm{Ind}^{W_F}_{W_J \cap W_F} \Big( \mathrm{Res}^{\s_{|F|}}_{W_{J} \cap W_F} V_{(1^{|F|})} \otimes \mathrm{Res}^{\s_{h-|F|}}_{W_{J}\cap W_F} \wedge^{k-|F|} V_{(h-|F|-1,1)} \Big) \nonumber\\
 = &(-1)^{k-|F|}\mathrm{Ind}^{W}_{W_F}\sum_{[J] \in D(F,H)/W_F} \mathrm{Ind}^{W_F}_{W_J \cap W_F} \Big( \mathrm{Res}^{\s_{|F|} \times \s_{h-|F|}}_{W_{J} \cap W_F} V_{(1^{|F|})} \otimes \wedge^{k-|F|} V_{(h-|F|-1,1)} \Big).\label{paving-term3-2}
\end{align}
As shown in \cite[Section~3]{karn2023equivariant}, we have
\begin{align*}
 D(F,H)/W_F \cong W_H \backslash C(F,H),
\end{align*}
where $C(F,H):=\{wF|w \in W ~ \text{and} ~ wF \subset H\}$.
Moreover, for every $J \in D(F,H)$, there exists a unique $T \in C(F,H)$ such that $W_J \cap W_F \cong W_H \cap W_T$.
Under this correspondence, we also have $|F|=|T|$.
Using these identifications, the expression in
\eqref{paving-term3-2} can be rewritten as
\begin{align} \label{paving-term3-identified}(-1)^{k-|F|}\sum_{[T] \in W_H \backslash C(F,H)}\mathrm{Ind}^{W}_{W_T \cap W_H} \Big( \mathrm{Res}^{\s_{|T|} \times \s_{h-|T|}}_{W_{H} \cap W_T} V_{(1^{|T|})} \otimes \wedge^{k-|T|} V_{(h-|T|-1,1)} \Big).
\end{align}
Letting $\mathcal{T}:=\{S| S \subset H , |S| \le k-2\}$. Substituting \eqref{paving-term3-identified} for
$\mathrm{Ind}^{W}_{W_F}\big(Q^{W_F}_{\M|_{F}}(t)\otimes (\mu^{W_F}_{\widetilde{\M}/F} - \mu^{W_F}_{\M/F})\big)
$
into \eqref{paving-term3-3}, we obtain
\begin{align*}
 &\sum_{[F] \in \{ S|S \subseteq H_{i}, |S| \le k-2, H_{i} \in \mathcal{H}\} /W} t^{\rk(\M/F)}~\mathrm{Ind}^{W}_{W_F}\Big( Q^{W_F}_{\M|_{F}}(t)\otimes (\mu^{W_F}_{\widetilde{\M}/F} - \mu^{W_F}_{\M/F}) \Big)\\
 &= \mathrm{Ind}^{W}_{W_H}\sum_{[T] \in \mathcal{T} /W_H} t^{k-|T|}\mathrm{Ind}^{W_H}_{W_T \cap W_H} \Big( \mathrm{Res}^{\s_{|T|} \times \s_{h-|T|}}_{W_{H} \cap W_T} V_{(1^{|T|})} \otimes \wedge^{k-|T|} V_{(h-|T|-1,1)} \Big).
\end{align*}
As in the analysis of the second term, we again use the standard fact that induction over orbits may be replaced by summation over all
elements.
Thus the above expression is equivalent to
\begin{align*}
&\mathrm{Ind}^{W}_{W_H}\sum_{T \in \mathcal{T}} t^{k-|T|}\mathrm{Res}^{\s_{|T|} \times \s_{h-|T|}}_{W_T \cap W_H} \Big( V_{(1^{|T|})} \otimes \wedge^{k-|T|} V_{(h-|T|-1,1)} \Big)\\
 =& \mathrm{Ind}^{W}_{W_H}\mathrm{Res}^{\s_h} _{W_H} \Big(\sum_{T \in \mathcal{T}} t^{k-|T|} V_{(1^{|T|})} \otimes \wedge^{k-|T|} V_{(h-|T|-1,1)} \Big)\\
 = &\mathrm{Ind}^{W}_{W_H}\mathrm{Res}^{\s_h} _{W_H}\sum^{k-2}_{i=0} \mathrm{Ind}^{\s_h}_{\s_{i} \times \s_{h-i}}\Big( V_{(1^{i})} \otimes \wedge^{k-i} V_{(h-i-1,1)} t^{k-i} \Big).
\end{align*}
This completes the proof.
\end{proof}

\begin{remark}
Elias, Miyata, Proudfoot, and Vecchi~\cite{elias2025categorical} introduced a framework 
of categorical valuative invariants for polyhedra and matroids. 
Within this framework, equivariant polynomial invariants, including the equivariant 
Kazhdan--Lusztig and $Z$-polynomials, admit categorical lifts that reflect their 
valuative properties. 
This approach enables equivariant computations for these invariants. 
It is natural to ask whether similar categorical techniques could be applied to the study of 
equivariant inverse $Z$-polynomials.
\end{remark}


We are now in the position to prove Theorem~\ref{thm-paving}.

\begin{proof}[Proof of Theorem \ref{thm-paving}]
Following \cite{ferroni2023stressed} and \cite{karn2023equivariant}, we compute the polynomial $y_{k,h}^{\s_h}(t)$ by considering the special matroid
$
 M_{k,h} := \U_{k-1,h} \oplus B_1,
$
equipped with the natural action of $\s_h$ on the
$\U_{k-1,h}$ component.
By \cite[Proposition~3.11]{ferroni2023stressed}, the uniform part $\U_{k-1,h}$ forms a stressed hyperplane of size $h$.
Relaxing this hyperplane yields the uniform matroid $\U_{k,h+1}$, so that
\begin{align}\label{YMrelax-YM}
 y^{\s_h}_{k,h}(t)
 =Y^{\s_h}_{\widetilde{M}_{k,h}}(t)-Y^{\s_h}_{M_{k,h}}(t)
 =\mathrm{Res}^{\s_{h+1}}_{\s_h} Y^{\s_{h+1}}_{\U_{k,h+1}}(t)
 -Y^{\s_h}_{M_{k,h}}(t).
\end{align}
Since the equivariant inverse Kazhdan--Lusztig polynomial is
multiplicative under direct sums, applying
\eqref{eq-ekl-uniform-1} from Proposition~\ref{prop-equi-uni} gives
\begin{align}\label{YM-1}
Y^{\s_h}_{M_{k,h}}(t)
 &= \Big(V_{(1)} + V_{(1)}t\Big) \Big( \sum_{i=0}^{\lfloor (k-1)/2 \rfloor} \sum_{x=0}^{i} \big(V_{(h-k+2,2^{x},1^{k-2x-2})} + V_{(h-k+3,2^{x-1},1^{k-2x-1})} \big) t^{i}\nonumber\\
 &\qquad + \sum_{i=0}^{\lfloor k/2-1 \rfloor} \sum_{x=0}^{i} \big(V_{(h-k+2,2^{x},1^{k-2x-2})} + V_{(h-k+3, 2^{x-1},1^{k-2x-1})} \big) t^{k-i-1} \Big).
\end{align}
For $0 \le i \le \lfloor \frac{k-1}{2} \rfloor$, define
\begin{align}\label{ci}
c_{i}:&=\sum_{x=0}^{i} \big(V_{(h-k+2,2^{x},1^{k-2x-2})} + V_{(h-k+3,2^{x-1},1^{k-2x-1})}\big)\nonumber\\
 &=\sum_{x=0}^{i} V_{(h-k+2,2^{x},1^{k-2x-2})} +\sum_{x=0}^{i-1} V_{(h-k+3,2^{x},1^{k-2x-3})}.
\end{align}
Then \eqref{YM-1} can be rewritten as
\begin{align*}
Y^{\s_h}_{M_{k,h}}(t)=
\begin{cases}
 c_{0}(1+t^{k})+\sum_{i=1}^{ \frac{k-1}{2} } (c_{i}+c_{i-1}) (t^{i}+t^{k-i})
 , &\text{if $k$ is odd}, \vspace{3mm} \\
 c_{0}(1+t^{k})+\sum_{i=1}^{ \frac{k}{2}-1} (c_{i}+c_{i-1}) (t^{i}+t^{k-i})+ 2c_{\frac{k}{2}-1}t^{\frac{k}{2}},
 &\text{if $k$ is even}.
 \end{cases}
\end{align*}
On the other hand,
\begin{align}\label{YM-relax-1}
 Y^{\s_h}_{\widetilde{M}_{k,h}}(t) &=\mathrm{Res}^{\s_{h+1}}_{\s_h} Y^{\s_{h+1}}_{\U_{k,h+1}}(t)\nonumber\\
 &= \sum_{i=0}^{\lfloor k/2 \rfloor} \sum_{x=0}^{i}\Big(\mathrm{Res}^{\s_{h+1}}_{\s_h} V_{(h-k+2,2^{x},1^{k-2x-1})} + \mathrm{Res}^{\s_{h+1}}_{\s_h}V_{(h-k+3,2^{x-1},1^{k-2x})} \Big) t^{i} \nonumber \\
 &\qquad+ \sum_{i=0}^{\lfloor (k-1)/2\rfloor} \sum_{x=0}^{i}\Big(\mathrm{Res}^{\s_{h+1}}_{\s_h} V_{(h-k+2,2^{x},1^{k-2x-1})} + \mathrm{Res}^{\s_{h+1}}_{\s_h}V_{(h-k+3,2^{x-1},1^{k-2x})} \Big) t^{k-i}.
\end{align}
By \cite[Lemma 4.1]{karn2023equivariant}, if $\lambda$ is a partition of $h+1$, then
\begin{align*}
 \mathrm{Res}^{\s_{h+1}}_{\s_h} V_{\lambda}=\bigoplus_{\lambda^{'}} V_{\lambda^{'}},
\end{align*}
where $\lambda'$ ranges over partitions of $h$ obtained from $\lambda$
by removing a single box.
In particular, we have for $0 \le x \le \lfloor \frac{k-1}{2} \rfloor$,
\begin{align}\label{RES-1}
 \mathrm{Res}^{\s_{h+1}}_{\s_h} V_{(h-k+2,2^{x},1^{k-2x-1})} =
 V_{(h-k+1,2^x, 1^{k-2x-1})} + V_{(h-k+2,2^{x-1},1^{k-2x})} + V_{(h-k+2,2^x,1^{k-2x-2})},
\end{align}
and for $1 \le x \le \lfloor \frac{k}{2} \rfloor$,
\begin{align}\label{RES-2}
 \mathrm{Res}^{\s_{h+1}}_{\s_h}V_{(h-k+3,2^{x-1},1^{k-2x})}= V_{(h-k+2,2^{x-1}, 1^{k-2x})} + V_{(h-k+3,2^{x-2},1^{k-2x+1})} + V_{(h-k+3,2^{x-1},1^{k-2x-1})}.
\end{align}
Substituting \eqref{RES-1} and \eqref{RES-2} into \eqref{YM-relax-1}, we obtain
\begin{align*}
Y^{\s_h}_{\widetilde{M}_{k,h}}(t)=
\begin{cases}
\sum_{i=0}^{ (k-1)/2 } d_i (t^{i}+t^{k-i}),
& \text{if $k$ is odd}, \vspace{3mm}\\
\sum_{i=0}^{ k/2-1 } d_i (t^{i}+t^{k-i}) \\
\quad
+ \big(
d_{k/2-1}
+ V_{(h-k+2,2^{k/2-1})}
+ V_{(h-k+3,2^{k/2-2},1)}
\big) t^{k/2},
& \text{if $k$ is even},
\end{cases}
\end{align*}
where
\begin{align}\label{di}
 d_{i}:=& V_{(h-k+2,2^{i},1^{k-2i-2})}
 + V_{(h-k+3,2^{i-1},1^{k-2i-1})} + \sum_{x=0}^{i} V_{(h-k+1,2^x, 1^{k-2x-1})}\nonumber\\
 &\qquad\qquad\qquad\qquad\qquad+ 3\sum_{x=0}^{i-1}V_{(h-k+2,2^{x},1^{k-2x-2})} + 2\sum_{x=0}^{i-2}V_{(h-k+3,2^x,1^{k-2x-3})}.
\end{align}

Now we prove that
\begin{align}\label{prf-ysh-eq}
 y^{\s_h}_{k,h}=Y^{\s_h}_{\U_{k,h}}(t)-\frac{1+(-1)^k}{2} V_{(h-k+2,2^{k/2 -1})}t^{\frac{k}{2}},
\end{align}
by considering the parity of $k$. Since $Y^{W}_{\M}(t)$ is palindromic, it suffices to compute the difference between the coefficients of $t^i$ in $Y^{\s_h}_{\widetilde{M}_{k,h}}(t)$ and $Y^{\s_h}_{M_{k,h}}(t)$ for $0 \leq i \leq \lfloor k/2 \rfloor$.
Suppose that $k$ is odd. Then
\begin{align*}
[t^{i}]Y^{\s_h}_{\widetilde{M}_{k,h}}(t)-[t^{i}]Y^{\s_h}_{M_{k,h}}(t)=
 \begin{cases}
 d_{i}-c_{i}-c_{i-1}
 , &\text{if $1 \le i \le \frac{k-1}{2}$ }, \\
 d_{0}-c_{0},
 &\text{if $i=0$}.
 \end{cases}
\end{align*}
For $i=0$, substituting \eqref{ci} and \eqref{di}, we have
\begin{align*}
 d_{0}-c_{0}=V_{(h-k+1,1^{k-1})}+V_{(h-k+2,1^{k-2})}-V_{(h-k+2,1^{k-2})}=V_{(h-k+1,1^{k-1})},
\end{align*}
which matches the constant term of $Y^{\s_h}_{\U_{k,h}}(t)$.
For $1 \le i \le (k-1)/2$, we have
\begin{align}\label{di-ci-ci-1}
 &d_{i}-c_{i}-c_{i-1}\nonumber\\
 =
 & V_{(h-k+2,2^{i},1^{k-2i-2})} + V_{(h-k+3,2^{i-1},1^{k-2i-1})}
 +\sum_{x=0}^{i} V_{(h-k+1,2^x, 1^{k-2x-1})}\nonumber
 \\
 &+ 3\sum_{x=0}^{i-1}V_{(h-k+2,2^{x},1^{k-2x-2})} + 2\sum_{x=0}^{i-2}V_{(h-k+3,2^x,1^{k-2x-3})}-\sum_{x=0}^{i} V_{(h-k+2,2^{x},1^{k-2x-2})}\nonumber
 \\
 &-\sum_{x=0}^{i-1} V_{(h-k+3,2^{x},1^{k-2x-3})}-\sum_{x=0}^{i-1} V_{(h-k+2,2^{x},1^{k-2x-2})} -\sum_{x=0}^{i-2} V_{(h-k+3,2^{x},1^{k-2x-3})}.
\end{align}
Note that
\begin{align}\label{h-k+2-eq}
 &V_{(h-k+2,2^{i},1^{k-2i-2})} + 3\sum_{x=0}^{i-1}V_{(h-k+2,2^{x},1^{k-2x-2})}
 -\sum_{x=0}^{i} V_{(h-k+2,2^{x},1^{k-2x-2})}
 \nonumber\\
 &\qquad\qquad\qquad\qquad\qquad\qquad\qquad-\sum_{x=0}^{i-1} V_{(h-k+2,2^{x},1^{k-2x-2})}=\sum_{x=0}^{i-1} V_{(h-k+2,2^{x},1^{k-2x-2})}
\end{align}
and
\begin{align}\label{h-k+3-eq}
 &V_{(h-k+3,2^{i-1},1^{k-2i-1})}
 + 2\sum_{x=0}^{i-2}V_{(h-k+3,2^x,1^{k-2x-3})}\nonumber
 \\
 &\qquad \qquad \qquad \qquad-\sum_{x=0}^{i-1} V_{(h-k+3,2^{x},1^{k-2x-3})}-\sum_{x=0}^{i-2} V_{(h-k+3,2^{x},1^{k-2x-3})}=0.
\end{align}
Substituting \eqref{h-k+2-eq} and \eqref{h-k+3-eq} into \eqref{di-ci-ci-1} yields
\begin{align*}
 d_{i}-c_{i}-c_{i-1}&=\sum_{x=0}^{i-1} V_{(h-k+2,2^{x},1^{k-2x-2})} +\sum_{x=0}^{i} V_{(h-k+1,2^x, 1^{k-2x-1})}\\
 &= \sum_{x=0}^{i} V_{(h-k+1,2^x, 1^{k-2x-1})}+V_{(h-k+2,2^{x-1},1^{k-2x})},
\end{align*}
which equals the coefficient of $t^{i}$ in $Y^{\s_h}_{\U_{k,h}}(t)$.
Thus, the case where $k$ is odd is verified.
Now suppose
that $k$ is even. Then
\begin{align*}
[t^{i}]Y^{\s_h}_{\widetilde{M}_{k,h}}(t)
-
[t^{i}]Y^{\s_h}_{M_{k,h}}(t)
=
\begin{cases}
d_{i}-c_{i}-c_{i-1},
& \text{if $1 \le i \le k/2-1$}, \\[4pt]
d_{0}-c_{0},
& \text{if $i=0$}, \\[6pt]
\begin{aligned}
& d_{k/2-1}
+ V_{(h-k+2,2^{k/2-1})} \\
& \quad
+ V_{(h-k+3,2^{k/2-2},1)}
- 2c_{k/2-1},
\end{aligned}
& \text{if $i=k/2$}.
\end{cases}
\end{align*}
For $0 \le i \le k/2-1$, the analysis is identical to odd $k$.
For $i = k/2$,
combining \eqref{ci} and \eqref{di}, we obtain
\begin{align*}
 &d_{k/2-1}+V_{(h-k+2,2^{k/2-1})}+V_{(h-k+3,2^{k/2-2},1)}-2c_{k/2-1}\\
 =&2V_{(h-k+2,2^{k/2-1})}+2V_{(h-k+3,2^{k/2-2},1)}+
 \sum_{x=0}^{k/2-1} V_{(h-k+1,2^x, 1^{k-2x-1})} + 3\sum_{x=0}^{k/2-2}V_{(h-k+2,2^{x},1^{k-2x-2})}\\
 &\quad+2\sum_{x=0}^{k/2-3}V_{(h-k+3,2^x,1^{k-2x-3})}
 -2\sum_{x=0}^{k/2-1} V_{(h-k+2,2^{x},1^{k-2x-2})}
 -2\sum_{x=0}^{k/2-2} V_{(h-k+3,2^{x},1^{k-2x-3})}\\
 =&\sum_{x=0}^{k/2-1} V_{(h-k+1,2^x, 1^{k-2x-1})}+\sum_{x=0}^{k/2-2}V_{(h-k+2,2^{x},1^{k-2x-2})}\\
 =&\sum_{x=0}^{k/2}\Big( V_{(h-k+1,2^{x},1^{k-2x-1})} + V_{(h-k+2,2^{x-1},1^{k-2x})}\Big)- V_{(h-k+2,2^{k/2-1})},
\end{align*}
which matches $[t^{k/2}] Y^{\s_h}_{\U_{k,h}}(t) - V_{(h-k+2,2^{k/2-1})}$ from Proposition~\ref{prop-equi-uni}.
Therefore, equation \eqref{prf-ysh-eq} holds for all $k$,
and the proof is complete.
\end{proof}

We now give a proof of Theorem \ref{cor-paving-or-z}.

\begin{proof}[Proof of Theorem \ref{cor-paving-or-z}]
We begin with the equivariant difference formula obtained in
Theorem~\ref{thm-paving}. Specializing that result to the trivial group and taking dimensions
on both sides of \eqref{eq-paving-thm}, we obtain
\begin{align}\label{coro-paving-eq1}
Y_{\widetilde{\M}}(t)-Y_{\M}(t)=Y_{\U_{k,h}}(t)-\frac{1+(-1)^k}{2} \operatorname{dim} V_{(h-k+2,2^{k/2 -1})} t^{\frac{k}{2}}.
\end{align}
Here $\widetilde{\M}$ is obtained from $\M$ by relaxing
a single stressed hyperplane of size $h \ge k$.
Recall that for any partition $\lambda \vdash n$,
the dimension of the irreducible
$\mathfrak{S}_n$-representation $V_\lambda$
equals the number of standard Young tableaux of shape $\lambda$.
By the hook-length formula \cite{frame1954hook}, we have
\begin{align}\label{coro-paving-eq2}
\operatorname{dim} V_{(h-k+2,2^{k/2 -1})} = \frac{4h}{(2h-k)(2h-k+2)}\dbinom{h-1}{\frac{k}{2},\frac{k}{2}-1,h-k}.
\end{align}
Substituting \eqref{coro-paving-eq2} into \eqref{coro-paving-eq1} yields
\begin{align}\label{coro-paving-eq3}
 Y_{\widetilde{\M}}(t)-Y_{\M}(t)=Y_{\U_{k,h}}(t)-\frac{2h\big(1+(-1)^k\big)}{(2h-k)(2h-k+2)} \dbinom{h-1}{\frac{k}{2},\frac{k}{2}-1,h-k}t^{\frac{k}{2}} .
\end{align}
Since
relaxing all stressed hyperplanes of size at least $k$
in a paving matroid of rank $k$ on $n$ elements yields
the uniform matroid $\U_{k,n}$,
applying \eqref{coro-paving-eq3} over all such relaxations
therefore gives the claimed formula.
\end{proof}

\begin{remark}
Let $\M$ be a sparse paving matroid of rank $k$ and cardinality $n$. Suppose $\M$
has $\lambda$ circuit-hyperplanes. Recall that for sparse paving matroids, the notions of stressed
hyperplanes and circuit-hyperplanes coincide.
Moreover, every stressed hyperplane has cardinality $k$.
Therefore, by Theorem~\ref{cor-paving-or-z}, we obtain
\begin{align*}
 Y_{\M}(t)=Y_{\U_{k,n}}(t)-\lambda \Big((1+t)^k- \frac{1}{2} \big(1+(-1)^k\big)C_{\frac{k}{2}}t^{\frac{k}{2}} \Big),
\end{align*}
where $C_n=\frac{1}{n+1}\binom{2n}{n}$ is the $n$-th Catalan number. This formula coincides with \cite[Theorem~1.3]{gao2025inverse}.
\end{remark}

We give a proof of Theorem \ref{thm-uni-bigest}.

\begin{proof}[Proof of Theorem \ref{thm-uni-bigest}]
Since $\M$ is paving, relaxing all stressed hyperplanes of size at least $k$
in $\M$ yields the uniform matroid $\U_{k,E}$.
Applying Theorem~\ref{thm-paving} at each relaxation step,
it suffices to prove that for every $h\ge k$,
each coefficient of $y^{\mathfrak S_h}_{k,h}(t)$ is an honest representation of
$\mathfrak S_h$.

	If $k$ is odd, then
	$y^{\mathfrak S_h}_{k,h}(t)=Y^{\mathfrak S_h}_{\U_{k,h}}(t)$,
	and the claim follows immediately from
	Proposition~\ref{prop-equi-uni}.
Assume therefore that $k$ is even.
By definition, the only potentially nontrivial coefficient of
$y^{\mathfrak S_h}_{k,h}(t)$ occurs in degree $k/2$.
It is thus enough to verify that
\begin{align}\label{pro-paving-eq1}
[t^{k/2}]\,Y^{\mathfrak S_h}_{\U_{k,h}}(t)
-
V_{(h-k+2,2^{k/2-1})}
\in \mathrm{Rep}(\mathfrak S_h).
\end{align}
	By Proposition~\ref{prop-equi-uni}, we have
	\begin{align*}
	[t^{k/2}]\,Y^{\mathfrak S_h}_{\U_{k,h}}(t)
&=\sum_{x=0}^{k/2}\big( V_{(h-k+1,2^{x},1^{k-2x-1})}+V_{(h-k+2,2^{x-1},1^{k-2x})}\big) \\
&=V_{(h-k+2,2^{k/2-1})}
 +\sum_{x=0}^{k/2-1}\big( V_{(h-k+1,2^{x},1^{k-2x-1})}+V_{(h-k+2,2^{x-1},1^{k-2x})}\big),
\end{align*}
where we adopt the convention that $V_\lambda=0$
if $\lambda$ is not a partition.
Consequently,
$$
[t^{k/2}]\,Y^{\mathfrak S_h}_{\U_{k,h}}(t)-V_{(h-k+2,2^{k/2-1})}
=\sum_{x=0}^{k/2-1}\big( V_{(h-k+1,2^{x},1^{k-2x-1})}+V_{(h-k+2,2^{x-1},1^{k-2x})}\big)
\in \mathrm{Rep}(\mathfrak S_h),
$$
which proves \eqref{pro-paving-eq1} and completes the proof.
\end{proof}

\vspace{4mm}

\noindent{\bf Acknowledgements.}
We thank Xuan Ruan for helpful discussions.
The first author is supported by the National Science Foundation for Post-doctoral Scientists of China (No.2020M683544).
The third author is supported by the National Science Foundation of China (No.12271403).

\bibliographystyle{alpha}
\bibliography{reference}

\end{document}